\documentclass[11pt]{amsart}
\usepackage{times}
\usepackage[T1]{fontenc}
\usepackage{amssymb, amsthm, amsmath}
\usepackage{mathrsfs}
\usepackage{mathtools} 
\usepackage[margin=1in]{geometry}

\usepackage{graphicx}

\usepackage{color}
\usepackage{enumerate}
\usepackage{multicol}

\usepackage{hyperref}

\usepackage[nameinlink]{cleveref}
\usepackage[all]{xy}
\usepackage{ marvosym }
\usepackage{comment}

\usepackage{breqn}

\usepackage{halloweenmath}
\usepackage{tikz-cd} 

\newtheorem{thm}{Theorem}[section]
\newtheorem{introtheorem}{Theorem}
\newtheorem*{thm*}{Theorem}

\newtheorem{prop}[thm]{Proposition}
\newtheorem*{prop*}{Proposition}

\newtheorem{lemma}[thm]{Lemma}
\newtheorem*{lemma*}{Lemma}

\newtheorem{proposition}[thm]{Proposition}

\newtheorem*{assm*}{Assumption}

\newtheorem{corollary}[thm]{Corollary}
\newtheorem*{corollary*}{Corollary}

\newtheorem{claim}[thm]{Claim}
\newtheorem*{claim*}{Claim}

\newtheorem{defi}[thm]{Definition}
\newtheorem*{defi*}{Definition}

\newtheorem{definition}[thm]{Definition}

\newtheorem{fact}[thm]{Fact}
\newtheorem*{fact*}{Fact}

\newtheorem*{example*}{Example}

\newtheorem*{examples*}{Examples}

\newtheorem*{conjecture*}{Conjecture}

\newtheorem*{goal*}{Goal}

\newtheorem*{subgoal*}{Subgoal}

\newtheorem*{question*}{Question}

\newtheorem*{problem*}{Problem}

\theoremstyle{definition}
\newtheorem{remark}[thm]{Remark}
\newtheorem*{remark*}{Remark}

\newenvironment{claimproof}[1]{\par\noindent\emph{Proof:}\space#1}{\hfill $\square_{Claim}$ \bigskip}

\DeclareMathOperator{\acl}{acl}
\DeclareMathOperator{\dcl}{dcl}

\DeclareMathOperator{\cl}{cl}

\DeclareMathOperator{\tp}{tp}

\newcommand{\RR}{\mathbb{R}}

\newcommand{\CK}{{\mathcal K}}

\newcommand{\CM}{{\mathcal M}}

\newcommand{\CO}{{\mathcal O}}

\newcommand{\CG}{{\mathcal G}}

\newcommand{\CY}{\mathcal Y}

\def\sub{\subseteq}

\def\Cc{\mathbb C}
\def\Kk{\mathbb K}
\def\Qq{\mathbb Q}

\def\CG{\mathcal G}

\def\0{\emptyset}
\def\Fr{\mathrm{Frob}}

\makeatletter
\newcommand{\tangent}{\mathrel{\vphantom{\cap}\mathpalette\@tangent\relax}}
\newcommand{\@tangent}[2]{%
  \vbox{\offinterlineskip
    \sbox\z@{$#1\pitchfork$}%
    \ialign{##\cr
      \raisebox{-0.37\height}[0pt][0pt]{%
        \resizebox{\ht\z@}{\height}{$\m@th\mspace{-1.5mu}-\mspace{-1.5mu}$}%
      }\cr
      \hidewidth\scalebox{1}[0.89]{$\m@th#1\cap$}\hidewidth\cr
    }%
  }%
}
\makeatother

\author{Assaf Hasson}
\address{Department of Mathematics, Ben Gurion University of the Negev, Be'er-Sheva 84105, Israel}
\email{hassonas@math.bgu.ac.il}

\author{Alf Onshuus}
\address{Departamento de matemáticas, Universidad de los Andes, Cra 1 No 18A-10, Bogotá, Colombia}
\email{aonshuus@uniandes.edu.co}
\author{Santiago Pinzon}
\address{Departamento de matemáticas, Universidad de los Andes, Cra 1 No 18A-10, Bogotá, Colombia}
\email{si.pinzon2355@uniandes.edu.co}

\begin{document}

\thanks{
The first author was supported by ISF grant No. 555/21.\\
The second author was supported by grant INV-2021-128-2318 from Facultad de Ciencias, Universidad de los Andes.\\
The third author was supported by Proyecto Semilla 2019 from Facultad de Ciencias, Universidad de los Andes.\\
This work is largely based on the third author's Ph.D. Thesis which was co-advised by the first and second authors.}

\title{Strongly minimal group relics of algebraically closed valued fields.} 

\maketitle 

\begin{abstract}
    We prove Zilber's trichotomy for reducts of ACVF expanding $(K,+)$ or $(K^*, \cdot)$. 
\end{abstract}

\section{Introduction}\label{introductionAndPreliminaries}

In  \cite{MaPi} Marker and Pillay prove that if $\CK$ is an algebraically closed field of characteristic $0$, and $\mathcal M$ is a strongly minimal reduct of $\CK$ expanding $(K,+)$, then either $\CM$ is locally modular, or $\mathcal M$ defines multiplication. In \cite{KR} Kowalski and Randriambololona prove the same result under the assumption that $\CK$ is an algebraically closed valued field of characteristic $0$. In that paper, the authors conjecture that a strongly minimal reduct, $\CM$,  of the full induced structure on some set $M$, definable in any algebraically closed valued field $\CK$, is either locally modular or interprets a copy of $\CK$. This question is an ACVF-variant of a similar question of Peterzil's in the o-minimal setting, who, in turn, echoed Zilber's Restricted Trichotomy Conjecture (for algebraically closed fields) of which the aforementioned Marker-Pillay result is a special case. 

As it turns out, the answer to the conjecture of Kowalski and Randriambololona, as formulated here, is positive. In the present paper, we make a first step toward proving this result. Proof of the full trichotomy, building on the present work, will appear in a subsequent paper by Castle, Ye and the first author. 

To formulate the results it is convenient, following \cite{CasHa}  to call a structure $\CM$ a $\CK$-relic if its universe, $M$ is $\CK$-interpretable (i.e., the quotient of a definable set by a definable equivalence relation) and so are all its basic relations and functions. By a $\CK$-group-relic, we mean a $\CK$-relic expanding a group. 

We extend the results of \cite{KR} in two main directions: first, we extend the result to positive characteristic, and then -- using the results in the additive case -- extend it further to strongly minimal expansions of the multiplicative group. Though these are \emph{definable} relics (whose universe is a definable subset of $K$), for technical reasons appearing 
in the extensions of the results to the multiplicative case 
we have to make a detour via a specific class of \emph{interpretable} relics (i.e., whose universe may, a priori, be an imaginary sort). So we prove a slightly more general result.   

To state our main result at the level of generality we need it, we introduce the following relation between definable groups:  

\begin{definition}
Let $\CK$ be any geometric expansion of a field, and let $(G, \oplus_G),(H, \oplus_H) $ be $\CK$-interpretable groups over some parameter set $A$. The group $G$ is \emph{locally equivalent} to the group $H$ if there are $a,b\in G$ and $\alpha, \beta\in H$  such that 
\begin{enumerate}
    \item each pair is independent generic  over $A$,
    \item $\acl(aA)=\acl(\alpha A)$,
    \item $\acl(bA)=\acl(\beta A)$ and 
    \item $\acl(a\oplus_G b A)=\acl(\alpha \oplus_H \beta A)$.
\end{enumerate}
We call   $(a,b,a\oplus_G b)$ and $(\alpha, \beta, \alpha\oplus_H \beta)$  as above \emph{equivalent group triples}. 
\end{definition}


\begin{introtheorem}\label{main}
     Let $\CK$ be an algebraically closed valued field, and  $\mathcal G:=(G,\oplus, \dots)$ a non-locally modular strongly minimal $\CK$-group-relic. Assume further that $G$ is locally equivalent to either $(K,+)$ or to $(K^*,\cdot)$.  Then $\mathcal G$ interprets a field, $F$, $\CK$-definably isomorphic to $(K,+,\cdot)$. Furthermore, the $\mathcal G$-induced structure on $F$ is that of a pure algebraically closed field.
\end{introtheorem}

The strategy of proof is similar, at its core, to that of the proof of the main result of \cite{MaPi}. As a replacement for the (basic) complex-analytic machinery used by Marker and Pillay we apply analogous tools from the theory of complete algebraically closed valued fields\footnote{A comprehensive development of most of the tools used in this text can be found in \cite{Cher}. In the text, we refer to published sources.}, as already appearing in \cite{KR}. Obviously, complete models of ACVF are not saturated, and since we do not know whether in such models analyticity is definable in families, some care is needed in formulating first order statements capturing the essence of analyticity sufficient for our needs. 

The Marker-Pillay strategy (building on ideas of Zilber and Rabinovich, see \cite{rabinovich} and references therein) boils down to finding in the relic $\mathcal G:=(G,\cdot, e, \dots)$ an infinite family $\{X_t\}_{t\in T}$  of 1-dimensional subsets of $G^2$ (below those are called \emph{plane curves} see Section \ref{ss plane curves} for more details) such that all $X_t$ go through $(e,e)$ and such that for some $\CK$-generic independent $s,t\in T$ there exists $u\in T\cap \acl_{\mathcal G}(s,t)$ for which $X_s\circ X_t$ is \emph{tangent} at $(e,e)$ to  $X_u$. 

Roughly, the notion of tangency referred to in the previous paragraph, can be explained as follows: the curves $X_t$ in the family contain, in a (valuative) neighbourhood of $(e,e)$ the graph of a $\CK$-definable differentiable function $f_t$, and $X_s\circ X_t$ is tangent to $X_u$ if $(f_s\circ f_t)'(e)=f_u'(e)$. In 
the additive setting, this is what is achieved.  I.e., a family $\{X_t\}_{t\in T}$ as above is constructed such that $\CG$ can detect equality of first derivatives (roughly, in the sense that if $X_t$ is tangent $X_s$ then $t$ forks with $s$ in the sense of $\CG$ . 

In the multiplicative setting, however, this is not the case due to the Frobenius  automorphism, and -- a priori -- the notion of tangency  detected by $\CG$ could be, using the same notation, $(f_s\circ f_t)^{(m)}(e)=f_u^{(m)}(e)$ for all $m$ smaller than some $n$ (here $f_s^{(m)}(e)$ denotes the $m$-th coefficient of a degree $m$ polynomial approximating $f_s$ near $e$). For deep technical reasons (to be explained in more detail in the text) detection of such higher order tangency is not quite sufficient for the construction of a field. Addressing this delicate point is the main technical difficulty of the paper. The strategy is, as in \cite{HS}, to deal first with the additive case, working -- again -- in a complete model, this time using the algebra of power series rather than appealing directly to geometric tools. In the multiplicative case we can show -- again using an idea from \cite{HS} --  that if the family of multiplicative translates of a curve $X$ (not intersecting any coset of any definable subgroup in an infinite set) fails to satisfy the desired conclusion then a commutative group $H$ locally equivalent to $(K,+)$ can be defined in $\CG$. 

In \cite{HS}, where the underlying field $K$ was a pure algebraically closed field, the classification of 1-dimensional algebraic groups allowed us to conclude that the $K$-connected component $H^0$ of the above-mentioned group $H$ is $K$-definably isomorphic to the additive group\footnote{In fact, in \cite{HS} it is only shown that $H^0$ is, up to a $\CK$-definable isomorphism, either $(K, +)$ or an elliptic curve.}. In the present paper we use a classification of 1-dimensional commutative groups definable in ACVF due to Acosta, \cite{acostaACVF}, combined with a result from  \cite{montenegroOnshuus} to carry out a similar, though somewhat more delicate analysis, to conclude the proof of our main theorem. 

We believe that the results of the current paper can be extended rather smoothly to analytic expansions of ACVF$_{0,0}$, e.g., in the sense of \cite{Lip}. Though the results of \cite{acostaACVF} and \cite{montenegroOnshuus} mentioned in the previous paragraph are not available in that setting, the resulting theory is 1-h-minimal, and the desired conclusion could be obtained using the tools developed in \cite{HaHaPeVF} using the theory of definable groups as in \cite{AcHa}. Such expansions will not, however, be considered in the present work.

\subsection{Structure of the paper}

In Section \ref{S: BaP} we recall and slightly develop some basic properties of ACVF, as well as basics of the analytic theory in complete models 
that we need. We also show that groups locally equivalent to $(K,+)$ or to $(K^*, \cdot)$ are definable, and set up the (simple) machinery for studying groups that are finite covers of affine groups. This last reduction is justified by the results of Section \ref{EOTS}. In Section \ref{BITIA} we develop basic intersection theory for definable curves in $G^2$. The idea is to exploit analytic properties of definable curves in  complete models to formulate first order statements that  hold in any model of ACVF and are strong enough for our needs. In Section \ref{FOSAAFFC} we use these result to reconstruct a field assuming we can construct a definable family with infinitely many first order slopes. The rest of the paper is dedicated to showing that such a family can always be constructed. Section \ref{PS} is dedicated to developing necessary tools from the theory of power series, which we then apply in Section \ref{TAC} to prove the main result in case $G$ is locally equivalent to $(K,+)$. Section \ref{TMC} is dedicated to the  multiplicative case. In the concluding section, we collect all the results of the paper to complete the proof of the main theorem.

\section{Background and preliminaries}\label{S: BaP}
Throughout $\CK$ denotes a pure  algebraically closed non-trivially valued field, that we assume to be $\aleph_1$-saturated, $K$ its underlying universe. We will often have to work in a complete elementarily equivalent model, that we denote $\Kk$. To keep notation simple, we will not distinguish between $\Kk$ and its universe. We also fix once and for all a strongly minimal group-relic, $\CG$. 

\textbf{Throughout, unless specifically stated otherwise, all model theoretic notions -- definability, independence, dimension, algebraic closure etc. -- refer to the full valued field structure on $\CK$ or $\Kk$.}

\subsection{Preliminaries on ACVF}
 Though the language we choose for $\CK$ is of little relevance, for concreteness, we single out two auxiliary sorts, $\Gamma$ for the valuation group and  $\textbf{k}$ for the residue field\footnote{The residue field will not play any role in the argument, up until the point where we have to identify the field $F$, whose construction is the goal of our main theorem.}. We work in the pure language of rings, augmented by a symbol $v$ for the valuation $v:K^*\to \Gamma$. We mostly use additive notation for the valuation, and as a matter of convenience let $v(0):=+ \infty$. In Section \ref{BITIA}, to conform with the relevant literature, and to enhance the complex analytic intuition, we revert, for a short while, to multiplicative notation.  

 We let $p=\text{char}(K)\geq 0$ and $\Fr:K\to K$ is the Frobenius automorphism $\Fr(x)=x^{p}$. If $p=0$ we let $\Fr$ be the identity.

We assume the reader to be familiar with the basic properties of the theory of algebraically closed valued fields (ACVF), but provide explicit references when needed. Interested readers are referred to, e.g, \cite[\S 7]{hru-haskell2005stable} (especially, sections 7.1-7.3) for a quick overview of relevant background. Let us mention one important such fact, originally due to Holly \cite{HolQE}. As stated here, it is Theorem 7.1 of \cite{hru-haskell2005stable}:

\begin{fact}\label{QEff} (Quantifier Elimination)
    $\mathrm{ACVF}_{p,q}$ has quantifier elimination in the above-mentioned language. 
\end{fact}

An open ball (centred at $x$ with valuative radius $r\in \Gamma$) in $K$ is the set $B_{>r}(x):=\{y\in K: v(x-y)>r\}$. Closed balls are defined analogously, with the strong inequality replaced by a weak inequality. Below, by ``ball" $B_r(x)$ we mean an open or a closed ball (of valuative radius $r$ and center $x$). Balls form the basis of a definable Hausdorff field topology on $K$. This topology, and the product topologies it induces on Cartesian powers of $K$, will be  referred to below as the \textit{valuation topology}. All topological notions (such as continuity, differentiability, convergence etc.) are understood -- unless explicitly stated otherwise -- with respect to the valuation topology. Since the topology is definable topological properties of definable objects are, as a rule, definable (uniformly), and whenever such definitions are clear, we will not dwell any further on them.  

It follows from quantifier elimination (this observation is also due to Holly, \cite{HolQE}) that definable subsets of $K$ (not of its Cartesian powers) are finite unions of \emph{Swiss Cheeses}, where a ``Swiss Cheese'' is a definable ball (or $K$ or a singleton) with a finite number of balls (or singletons) removed. We need the following well known generalization, which, although it appears in \cite{KR}, we give the proof for completeness: 

\begin{prop}\label{decompositionProp}
Suppose $Y\subseteq  K^m$ is $\CK$-definable and infinite. Then $Y$ can be written as a finite union of  subsets of $K^m$ that are relatively open subsets of irreducible Zariski closed sets.

If $Y\subseteq K\times K$ does not have isolated points, one can write $Y$ as a finite union $$Y=\bigcup\{(x,y)\in V_i:L_i(x,y)=0\}$$ where $L_i(x,y)$ is an irreducible polynomial, $V_i\subseteq K\times K$ is open and $\{(x,y)\in V_i:L_i(x,y)=0\}\neq \emptyset$. 
Moreover, if $V_i\cap L_i$ is infinite, then $L_i$ is definable over the algebraic closure of a code for $Y$. 
\end{prop}

   \begin{proof}

   This is very similar to the proof of Lemma 3.4 of \cite{KR}.
   
Using Fact \ref{QEff}  write $Y$ as a finite union of sets of the form:

\begin{equation}\label{QEDecompositionEcuation}
\bigcap_i \{x\in K^m: v(f_i(x)) \Box_i v(g_i(x))\}
\end{equation}
where $\Box_i\in \{<, =\}$ and each $f_i$ and $g_i$ are polynomials.
For each $i$ let 
$$A_i:=\{x\in K^m:v(f_i(x))\Box_i v(g_i(x))\}.$$

If $\Box_i$ is $<$ and $g_i$ is not the constant function $0$, then 
$$A_i=\{x\in K^m: g_i(x)\neq 0\text{ and }v(f_i(x)/g_i(x))< 0\}\cup \{x\in K^m: f_i(x)\neq 0 =g_i(x)\}$$ which is the  intersection of an open (as the polynomials are continuous in the valuation topology) with a Zariski closed. If $\Box_i$ is $<$ and $g_i$ is constant $0$ then 
$$A_i=\{x\in K^m:v(f_i(x)) < \infty\}=\{x\in K^m:f_i(x)\neq 0\}$$ which is open.

If $\Box_i$ is $=$ and both, $f_i$ and $g_i$ are not the zero polynomial, then
$$A_i=\{x\in K^m:g_i(x)\neq 0\text{ and }v(f_i(x)/g_i(x))=0\}$$ which is open.

Finally, if $\Box_i$ is $=$ and either $f_i$ or $g_i$ is the zero polynomial, then $A_i$ is Zariski closed. 
Therefore, the intersection of Equation \ref{QEDecompositionEcuation} is a Zariski closed $C$ set intersected with an open set $U$, take the irreducible decomposition of $C$:
$$C=C_1\cup\ldots\cup C_d$$ so $$C\cap U=(C_1\cap U\cap V_1)\cup \ldots\cup (C_d\cap U\cap V_d)$$ where $$V_i=K\times K\setminus \left(C_1\cup C_2\cup\ldots\cup C_{i-1}\cup C_i\cup\ldots\cup C_d\right).$$ 
Then each $C_i\cap U\cap V_i$ is an open subset of $C_i$ that is its Zariski closure so the intersection of Equation \ref{QEDecompositionEcuation} has the desired form and then $Y$ has also the desired form.

The proof of the last part of the proposition is clear, and left to the reader.\end{proof}

There are two useful consequence of Proposition \ref{decompositionProp} that we will often use without further reference. The first is that any infinite definable subsets of  $K$ has non-empty interior. The second is that  ACVF is \emph{algebraically bounded}, i.e., in ACVF the model theoretic algebraic closure coincides with the field theoretic algebraic closure.  In particular, $\acl(\cdot)$ satisfies the Exchange Property, implying that $\acl(\cdot)$ induces a dimension on $K$. To recall, $\dim(a/A)=k$ if and only if there is $a'\subseteq a$ of size $k$ that is $\acl$-independent over $A$ and $a\in \acl(a')$. This extends to definable sets by setting \[
\dim(X)=\max\{\dim(a/A): a\in X, X \text{ is $A$-definable}\}
\]
(where the last definition should be evaluated in an $\aleph_0$-saturated model). Given a set of parameters $A$ and an $A$-definable set $X$, an element $a\in X$ is \emph{generic over $A$} if $\dim(a/A)=\dim(X)$.

The following is now immediate: 
\begin{lemma}\label{L: open function}
    Let $B\sub K$ be a ball, and $f: B\to K$ a definable function. If $f$ is continuous with finite fibres, then $f$ is open. 
\end{lemma}
\begin{proof}
    By quantifier elimination, $f(B)$ is a finite union of Swiss Cheeses. Since it is continuous with finite fibers, the image cannot have isolated points, so it is open (because every Swiss Cheese which is not a point is). Since the same argument holds of any ball $B'\sub B$, the conclusion follows. 
\end{proof}

We use without further reference the following two facts: 
\begin{fact}
    Let $X\sub K^n$ be a definable set and $x\in X$ a generic point. Then for any open, $U\ni x$ there is an open $W\sub U$ containing $x$ such that $x$ is generic in $W\cap X$. 
\end{fact}
\begin{proof}
    This is immediate, since $\CK$ has Exchange: given $U$ as in the statement, we may replace $U$ with a ball centred at $x$. Let $r$ be the valuative radius of $U$. Let $x'$ be such that $v(x-x')=r'>r$ and $x'$ is generic independent of $x$. Let $r<r''<r'$ be generic, independent of $x$ over $x'$. Then the ball of radius $r''$ centred at $x'$ will satisfy the statement. 
\end{proof}

We also need the following: 
\begin{fact}\label{F: gen to ball}
    Let $S$ be a definable set $s\in S$ generic. Assume that $\tp(s)\vdash \phi(x)$ for some formula $\phi(x)$. Then there is an open set $U$ such that $\models \phi(s')$ for all $s'\in U\cap S$. 
\end{fact}
\begin{proof}
    By quantifier elimination and compactness. 
\end{proof}

\subsection{Plane curves in strongly minimal structures}\label{ss plane curves}
As is common in proofs of Zilber's Trichotomy in various settings, plane curves are the main object used for reconstructing a field. We give a quick overview of the main notions appearing below. Since these have. by now, become quite standard, we will be brief, referring interested readers to the introductory sections of \cite{CAS} (and references therein) for a more detailed discussion.

Let $\CG$ be a strongly minimal expansion of a group. By a \emph{plane curve},  we mean a definable  set $C\sub G^2$ of Morley Rank 1. We do not require $C$ to be strongly minimal. We will mostly be interested in definable families of plane curves, i.e., definable sets $\psi(x,y)$ such that $\psi(x,a)$ is a plane curve for all $a$ for which $\psi(x,a)$ is consistent.  For greater readability, we mostly denote definable families of plane curves as $\{Y_s\}_{s\in S}$. A definable family of plane curves $\{Y_s\}_{s\in S}$ is \emph{almost faithful} if for all $s\in S$ the set $\{t\in T: |Y_s\cap Y_t|=\infty\}$ is finite. It is a standard easy exercise to show that in a strongly minimal structure every definable plane curve is, up to finite symmetric difference, generic in some almost faithful family of plane curves (and, by weak elimination of imaginaries, the parameter set of the family may be assumed definable -- though this is not needed in the sequel). 

\begin{remark}\label{R: finite intersection}
    Note that if $\{C_s\}_{s\in S}$ is an almost faithful family of plane curves, then the canonical base of any strongly minimal $C\sub C_s$ is inter-algebraic with $s$. In particular, if $\{D_t\}_{t\in T}$ is any other almost faithful family of plane curves and $\mathrm {RM}(T)\neq \mathrm{RM}(S)$ then for generic $s\in S$ and $t\in T$ the intersection $C_s\cap D_t$ is finite. 

\end{remark}

If $\{Y_s\}_{s\in S}$ is an almost faithful family of plane curves, the dimension of the family is defined to be the Morley Rank of $S$. A strongly minimal structure is non-locally modular if and only if it admits a 2-dimensional definable (almost faithful) family of plane curves (see, e.g., \cite[Theorem 3.4.2]{Hr4}). \textit{For the present paper, this can be taken to be the definition of non local modularity}. 

In the sequel, when $\CG=(G,\oplus,\ldots)$ is an ACVF group-relic, it is useful to have good control over various properties that may not be definable. E.g., when working in a complete model, we would like to have definable families of curves that are locally analytic at every point (this will be defined below). In non-locally modular strongly minimal groups there is a standard and easy way of constructing arbitrarily large definable families of plane curves, in a way giving us sufficient control over the geometric properties of curves in the resulting family to achieve this goal. 

Given a plane curve $C\sub G^2$ and $a:=(a_1,a_2) \in G^2$ we let 
\[
t_a(C):=\{(x\ominus a_1, y\ominus a_2): (x,y)\in C\}
\]
be the $a$-translate of $G$. If $C$ is strongly minimal $t_a(C)$ is, obviously, also strongly minimal, and $t_a(C)\cap t_b(C)$ is infinite if and only if $a\ominus b$ belongs to the stabilizer of the generic type of $C$, denoted $\mathrm{Stab}^*(C)$.  Thus, the family of translates $\{t_a(C): a\in G^2\}$ is $2$-dimensional if and only if $\mathrm{Stab}^*(C)$ is finite. 

A strongly minimal plane curve $C$ is \emph{affine}\footnote{This terminology refers to the analogous notion from linear algebra, and not from  algebraic geometry.} if $\mathrm{Stab}^*(C)$ is infinite (which is readily seen to be equivalent to $C$ coinciding, up to finite symmetric difference, with a coset of a proper definable subgroup of $G^2$). It follows from \cite[Theorem 7.2]{castleHasson} (and is well known) that $\CG$ is not locally modular if and only if it admits a non-affine strongly minimal plane curve. Note that by what we have just said, non-local modularity of the relic $\CG$ implies the existence of a 2-dimensional definable family of plane curves (\textit{vis.}, the family of translates  of a non-affine plane curve). 

Families of translates of non-affine plane curves will play a crucial role in our construction. We remind some construction techniques of plane curves that are useful in the present work: 
\begin{itemize}
    \item If $C$ is a plane curve then $C^{-1}$, the opposite curve, is $\{(y,x): (x,y)\in C\}$. 
    \item For plane curves $C_1,C_2$ their (functional) sum is \[C_1\oplus C_2:=(x,y_1,\oplus y_2): (x,y_1)\in C_1, (x,y_2)\in C_2\}.\] 
    \item If $C_1,C_2$ are plane curves their (functional) composition $C_1\circ C_2$ is \[\{(x,y): \exists z\left((x,z)\in C_1\land (z,y)\in C_2\right)\}.\]
\end{itemize}

Following \cite{CAS} we call a plane curve \emph{trivial} if one of its projections has an infinite fibre. It is a standard easy exercise to verify that if $C$ is non-trivial neither is any of its translates nor is its opposite curve. Moreover, if $C_1, C_2$ are non-trivial then their composition is itself a non-trivial curve.  \textbf{From now on we tacitly assume that all definable curves appearing in the text are non-trivial. 
}

It is well known (see, e.g., (the proof of) \cite[Lemma 3.20]{HE} for an argument) that if $\mathcal C:=\{C_s\}_{s\in S}$ is an almost faithful family of plane curves with $\mathrm {MR}(S)\ge 2$ then either the composition family $\mathcal C\circ \mathcal C:=\{C_s\circ C_t: s,t\in S\}$ has rank\footnote{The family $\mathcal C\circ \mathcal C$ need not, of course, be almost faithful. By its rank, we mean the rank of the canonical base of any strongly minimal $C$ contained in a generic member of the composition family.}  higher than $\mathrm{RM}(S)$ or an infinite field is interpretable. Since, in the latter case, our task is achieved, we tacitly assume, throughout, that the former case occurs. In particular, by what we have said above, $C_s\cap (C_t\circ C_r)$ is finite for $s\in S$ generic and $t,r\in S$ independent generics.  

\subsection{A first reduction}\label{ss: definable}
We fix, up to the end of Section \ref{ss: definable}, a strongly minimal non-locally modular 1-dimensional group relic $\CG$, that is locally equivalent to $(K, +)$ or to $(K^*, \cdot)$. We show that $G$ is definable.

The following lemma is a simplified version of \cite[Corollary 9.5]{HY}. This is the only place in the paper where dp-rank is invoked, so we do not provide a definition. We point out only that any model of ACVF has dp-rank 1, \cite{DoGoLi}, that $\mathrm{dp-rk}(X)\ge \mathrm{dp-rk}(f(X))$ for any definable function $f$, and equality holds if $f$ has finite fibres. Moreover, $\mathrm{dp-rk} (X \times Y)=\mathrm{dp-rk}(X)+\mathrm{dp-rk}(Y)$ (see \cite[\S 4.2]{Si} for the definition and reference). All of this means that any interpretable set has finite dp-rank, so -- in particular -- so does $G$. 
An element  $a\in G$ is \emph{generic over a parameter set $A$} if $\mathrm{dp-rk}(a/A)=\mathrm{dp-rk}(G)$. We note that restricted to definable subsets of $K$ the dp-rank coincides with the $\acl$-dimension we have defined above, and thus also the associated notions of genericity coincide. 

We prove a simplified version of \cite[Corollary 9.5]{HY}: 
\begin{lemma}
    There exists a definable finite-to-one map from $G$ to $K^n$ (some $n$). 
\end{lemma}
\begin{proof}
     Since $G$ is locally equivalent to $(K,+)$ or to $(K^*, \cdot)$ there exist a parameter set $A_0$, an element  $a\in G$ and some $a'\in K$ such that $a$ and $a'$ are generic and $\CK$-inter-algebraic over $A_0$. For simplicity of notation, assume $A_0=\emptyset$. 
     
     By compactness, there is a formula $\phi(x,y)$ (over $A_0$) witnessing this. Let $C_0$ be its set of realizations, and $B_0$ its projection to the first coordinate.  Since $a\in B_0$  and $a$ is generic, $B_0$ is infinite. Note also that $G$ is dp-minimal (i.e., $\mathrm{dp-rk}(G)=1)$, because $a\in G$ is generic and $\mathrm{dp-rk}(a)=1$, by virtue of its inter-algebraicity with $a'$. 

    By the main result of \cite{HrACVFQE} ACVF has elimination of imaginaries up to the so-called, \emph{geometric sorts}. So $G\sub K^n \times S$ where $S$ is a product of  geometric sorts  of \cite{HrACVFQE}. We will show that the projection $G\to K^n$ has finite fibres. 
    
    Assume toward a contradiction that this is not the case. Let $Y\sub G$ be an infinite fibre of the projection. Then $Y$ can be identified with a definable subset of $S$. 
    
    Since ACVF is geometric (i.e., $\acl(\cdot)$ satisfies the Exchange Property and definable families of finite sets are uniformly bounded) it follows from Gagelman's work \cite{Gagelman}, that the $\acl$-dimension extends from definable to interpretable sets. The only properties that we need are that this extended dimension is additive, and that if $Y$ is definable and $E$ is a definable equivalence relation all of whose classes are open, then $\dim(Y/E)=0$. In particular, all the geometric sorts are 0-dimensional, and by additivity $\dim(S)=0$. 

    Consider the function $f: Y\times B_0\to G$ given by $f(y,b)=yb$ (where multiplication is taken in $G$). Since $G$ is dp-minimal, and $Y,B_0$ are infinite, additivity of the rank implies that some fibre of $f$ is infinite. Such a fibre defines a bijection between an infinite subset of $Y$ and an infinite subset of $B_0$, but $B_0$ is in finite-to-finite correspondence with a definable subset of $K$, so any infinite subset of $B_0$ has dimension 1, whereas $\dim(Y)=0$, so there cannot be a definable finite-to-finite correspondence between any infinite subset of $B_0$ and a subset of $Y$, a contradiction.\end{proof}

We thank J.P. Acosta for providing us with the next lemma. The proof presented here was suggested to us by B. Castle: 
\begin{lemma}
    If $X$ is interpretable in ACVF and there is a definable finite-to-one function $f: X\to K^n$ (some $n$) then $X$ is definable. 
\end{lemma}
\begin{proof}
    After naming a single constant $c\notin \CO$, we may assume, using the completeness of ACVF$_{p,q}$ that for all $A\sub K$ the algebraic closure $\acl(A)$ (more precisely $\acl(A)\cap K$) is a model. Let $X=Y/E$ for $Y\sub K^n$ a definable set and $E$ a definable equivalence relation on $Y$.  By assumption $\acl(X)\cap K$ is a model, and since $f: X\to K^n$ has finite fibres $X\sub \acl(X)\cap K$, and, in fact, since $x\in Y/E$ implies $x\in \acl(K)$ it follows that $X\sub \dcl(\acl(X)\cap K)$. As a result, we can replace $Y$ with some $Y'$ such that $E$ has finite classes in $Y'$ and $Y'/E=X$. Since $K$ is a field, it eliminates finite imaginaries, i.e., we can find an injection $h: X\to K^m$ (some $m$), with the desired conclusion. 
\end{proof}

\begin{remark}
    In the above argument, elimination of imaginaries for ACVF is invoked only to conclude that $G$ is in finite-to-finite correspondence with a definable subset of $K^n$. In the applications, we only need this reduction for groups that we know, a priori, to have this property. 
\end{remark}

Combining the last two lemmas, we see that $G$ is definable, as claimed. Consequently, \textbf{from now on, we assume, throughout, that $G$ is a definable group.  }

\subsection{Complete models and analytic functions}\label{ss; complete models}
Recall that a valued field is complete if its value group is Archimedean (i.e., embeds into $(\mathbb R, +, <)$) and is complete with respect to the metric induced by the valuation. 
Conveniently, every completion of ACVF has a complete model (in a less overloaded terminology, every saturated model of ACVF has an elementary submodel that is complete as a valued field). Indeed, using the completeness of $\mathrm{ACVF}_{p,q}$ we get that $\Cc[[t^\Qq]]\models \mathrm{ACVF}_{0,0}$ and similarly $\mathbb  F_p^{alg}[[t^{\Qq}]]\models \mathrm{ACVF}_{p,p}$. Clearly, $\Cc_p\models \mathrm{ACVF}_{0,p}$.  For our purposes, the main advantage of working in a complete model is the theory of analytic functions in complete local fields, that will allow us, ultimately, to develop the intersection theory we need in ACVF. Below, and throughout the paper, if $\CK\models \mathrm{ACVF}$ we let $\Kk$ denote a complete elementarily equivalent model.

Recall that in a  non-Archimedean setting, a series $\sum\limits_n a_n$ converges if and only if $\lim\limits_n v(a_n)=\infty$.  The following definition is standard:
\begin{definition}
    Let $\Kk$ be a complete valued field, $U\sub K$ an open set. A function $f: U\to \Kk$ is analytic at $z_0\in U$ if on some ball $B\ni z_0$ the function $f$ coincides with a power series $\sum\limits_{n=0}^\infty c_n(z-z_0)^n$ converging in $B$. 
\end{definition}

\textbf{For the purposes of this paper, a function is analytic on a ball $B$ if it coincides on $B$ with a convergent power series.} In the literature, the definition is more general, allowing the function to coincide with a Laurent series converging on its domain, but since those will not be used in the sequel, we keep a slightly more restrictive definition. We note that a function analytic at every point of its domain need not be analytic. We call such functions \emph{locally analytic}.

Complete fields satisfy the identity theorem (see \cite[Theorem II.10.5.2]{LAG}):

\begin{fact}\label{identityTheoremf} (Identity Theorem) Let $$f(x)=\sum_{n\geq 0} a_n x^{n}$$ be a power series converging in a neighborhood $D(f)$ of $0$. Assume that $f(x)=0$ for all $x$ in some non-empty open set $U\subseteq D(f)$, then $a_n=0$ for all $n\textit{}$.
\end{fact}

We also have the following version of the implicit function theorem (see, e.g., \cite[Theorem II.10.8]{LAG}): 

\begin{fact}(Implicit Function Theorem)\label{implicitFunctionTheoremf}
Suppose $\Kk$ is complete, $U\subseteq \Kk^n$ is open and $F:U\to \Kk$ is an analytic function at $z\in U$. Assume $z$ is such that $\frac{\partial F}{\partial x_n}(z)\neq 0$, then there are $U_1\subseteq \Kk^{n-1}$ and $U_2\sub K$ both open, $U'\subseteq U$ open with $z\in U' \cap (U_1\times U_2)$ and $f:U_1\to U_2$ analytic such that:

\begin{equation*}
    \{u\in U':F(u)=0\}=\{(x,f(x)):x\in U_1\}.
\end{equation*}
\end{fact}

As a corollary, we have:

\begin{fact}\label{inversionf}(Inverse Function Theorem)
Suppose $\Kk$ is complete, $U\subseteq \Kk$ is open and $F:U\to \Kk$ is an analytic function at $z\in U$. Assume $F(z)=z'$ and $\frac{\partial F}{\partial x} (z)\neq 0$ then there is a unique analytic function $G$ converging in a neighborhood of $U'\ni z'$ such that $F(G(y))=y$ for each $y\in U'$.
\end{fact}

A key lemma now is: 

\begin{lemma}\label{L: t-an}
Let $\Kk\models\mathrm{ACVF}$ be a complete model,  
 $f: \Kk\to \Kk$ a definable function. Then there exists a finite set $S_f\sub \Kk$ definable over the same parameters as $f$, such that on $\Kk\setminus S_f$ the function $f$ is 
a composition of a power of the Frobenius automorphism with a locally analytic function. Moreover, if $\{f_s\}_{s\in Q}$ is a $\Kk$-definable family of unary functions there is a uniform bound $N$, such that $f_s=\Fr^N\circ g_s$ for $g_s$ locally analytic in $K\setminus S_{f_s}$. 

\end{lemma}

\begin{proof}

Let $f$ be a definable unary function. Proposition \ref{decompositionProp} implies that the Zariski closure of the graph of $f$ coincides with the zero set of a polynomial, 
$F(x,y)\in K[x,y]$. It will suffice to prove the result for each irreducible component of $F(x,y)$ separately. Indeed, if the graph of  $f$, coincides on an irreducible component $C_i$ of $F$ with $\Fr^{n_i} f_i$ for some locally analytic function $f_i$ and $|n_i|$ minimal possible, then by taking $n=\min\limits_i{n_i}$ and 
$g_i:=\Fr^{n_i-n}\circ f_i$ the choice of $n$ assures that the $g_i$ are locally analytic, and $f$ is locally the composition of 
$\Fr^n$ with $g_i$.  

So we may assume that $F(x,y)$ is irreducible. By the Implicit Function Theorem, the following claim is enough to complete the proof:

\begin{claim}
Let $F(x,y)\in K[x,y]$ be an irreducible polynomial and suppose that $F$ is not constant. 
  then $F(x,y)=H(x,\Fr^n(y))$ for some $n\in \mathbb N$ and some polynomial $G$ such that only finitely many points $(a,b)$ satisfy 
   \begin{equation}
     H(a,b)= \frac{\partial H}{\partial y}(a,b)=0.
 \end{equation}

\end{claim}

\begin{claimproof}\label{C: claim Fr}
If there are only finitely many points $(a,b)\in K^2$ such that 
 \begin{equation}
     F(a,b)= \frac{\partial F}{\partial y}(a,b)=0
 \end{equation} 
 then $n=0$ and $F=H$.

 So assume otherwise. Notice that it is enough to show that whenever one has such a polynomial, $F(x,y)$ there is some polynomial $G(x,y)$ with $F(x,y)=G(x,\Fr(y))$. Because the $y$-degree of $G$ is smaller than that of $F$, and we may conclude by induction.

Since $F(x,y)=0$ and $\frac{\partial F}{\partial y}[x,y]=0$ define 1-dimensional algebraic sets and  $F$ is irreducible and non-constant, they can only have infinite intersection if the latter contains the former. Since $\deg(\frac{\partial F}{\partial y})<\deg(F)$ this can only happen if $\frac{\partial F}{\partial y}[x,y]$ is the zero polynomial. This proves the lemma in characteristic $0$. 

In positive characteristic, if $$F(x,y)=f_0(x)+f_1(x) y +\ldots+f_n(x) y^n$$ with $f_i(x)\in K[x]$ for $i=0,1,\ldots, n$. As $$\frac{\partial F}{\partial y}(x,y)=f_1(x)+2f_2(x)y+\ldots+nf_n(x)y^{n-1}=0$$ for any $a$ and any $i=1,\ldots, n$ we have that $if_i(a)=0$. So either $f_i$ is zero or $p\mid i$ which implies that $F(x,y)$ is a polynomial in the variable $y^p$, as required.   
\end{claimproof} 

Fixing $H(x,z)$ as in the claim we get, by the Implicit Function Theorem, that away from the finite set $S_f$ of those points  $(x_0,y_0^{p^n})$ where $\frac{\partial H}{\partial y}(x_0, y_0^{p^n})=H(x_0,y_0^{p^n})=0$, there exists an analytic function $\psi_{a,b}$ defined on an open neighbourhood of $a$ such that on its domain  $H(x,\psi_{a,b}(z))=0$. Thus, by the definition of $H$, in a neighbourhood of  the point $a$ the graph of $f$ coincides with $\Fr ^{-n}\circ \psi_{a,b}$. Since $\psi_{a,b}$ is analytic, $F$ and thus also  $H$ are defined over the same parameter set as $f$ was so is $S_f$, and the conclusion follows. 

For the ``moreover'' part notice, first, that -- as mentioned above -- to each $f_s$ the polynomial $H_s$ provided by the claim applied to $f_s$ is defined over $s$. Thus, there is a uniform bound on  $\{\mathrm{deg}(H_s): s\in Q\}$. As the power of Frobenius used for $f_s$ is bounded by the degree of $H_s$ we get the desired result. 
\end{proof}

In the sequel, we are mostly interested in topological properties of definable functions arising from questions on the intersections of definable curves in families. So it will be useful to fix some \textit{ad hoc} terminology: 
\begin{definition}
    Let $\Kk$ be a complete model of ACVF. A definable function $f:U\to \Kk$ is $t$-analytic if for some $n\in \mathbb N$ the function $\Fr^n\circ f$ is  locally analytic in $V$.   
\end{definition}

In this terminology we have shown that, every definable function in a complete model of ACVF is, away from finitely many points, $t$-analytic. To translate such results as, say, the Implicit Function Theorem to elementary extensions we need some sort of definability of analyticity (in families). We do not know of such a result in the context of ACVF, but the following weaker version is sufficient for our needs:

\begin{lemma}\label{L: generically locally analytic}
    Let $\Kk$ be a complete model, $C\sub \Kk^2$ an  $A$-definable 1-dimensional set. Then there exists a finite set $S\sub \acl(A)$ such that for each $c\in C\setminus S$ there exists a  $t$-analytic function whose graph coincides with $C$ in a neighbourhood of $c$. 
\end{lemma}
\begin{proof}
    By Proposition \ref{decompositionProp} $C$ is locally Zariski closed. First, remove the (finite) set of intersections between different irreducible components of the Zariski closure of $C$. Then, working in each component, it is enough if we prove the result assuming that $C$ is the intersection of a ball with the zero-set of an irreducible polynomial $F(x,y)$. We may further assume that $F$ is non-constant (as otherwise there is nothing to prove). Let $H$ and $n$ be as provided by Claim \ref{C: claim Fr}. So $F(x,y)=H(x,\Fr^n(y))$.  The set of points of the curve $H(x,z)=0$ for which $\frac{\partial H}{\partial z}\neq 0$ is definable and co-finite, and  we can apply the Implicit Function Theorem, \ref{implicitFunctionTheoremf} to any such point of $H(x,z)=0$. Composing the resulting function with $\Fr^{-n}$ shows that $F(x,\Fr^n(y))=0$ is $t$-analytic on a co-finite set, as claimed.   Taking 
    \[
    S:=\left \{(a,b^{-p^m}): H(a,b), \frac{\partial H}{\partial y}(a,b)\neq 0\right \}
    \] 
    satisfies the required properties, since -- if $C$ is irreducible --  the polynomial $H$ and the integer $n$, are invariant under any automorphism (of a saturated elementary extension) fixing $C$. 

\end{proof}

Since the set $S$ provided in the conclusion of the previous lemma is definable (over the same parameters as the curve), whenever $P$ is an elementary property shared by all unary analytic functions (e.g., open, locally constant if and only if constant, $n$-times differentiable for all $n$, etc.) we have the following elementary version of the Implicit Function Theorem: 
\begin{corollary}
    Fix an elementary property $P$ of definable analytic functions.  Let $\CK\models \mathrm{ACVF}$ and $C\sub K^n$ be an $A$-definable curve. Then there is some finite set $S\subseteq \acl(A)$ such that if $c\in C\setminus S$ then there exists a ball $B\ni c$ such that $C\cap B$ is the graph of a definable function with property $P$. 
\end{corollary}

In fact, since the valuation topology is definable, the same is true of type definable properties of analytic functions. This will not, however, be used in the present text. 

\subsection{The setting}\label{ss: the setting}
The tools developed in Section \ref{ss; complete models} as well as in Section \ref{BITIA} below are described in the affine context: we consider (mostly) definable subsets of $K$ and $K^2$ and their geometric properties.  All this work, and in particular our study of definable (partial) functions $f:K\to K$ and of 1-dimensional subsets of $K^2$ extends in the usual way to definable 1-dimensional manifolds. Local notions such as, continuity, openness (of mappings), differentiability, tangency, and analyticity (in complete models) extend automatically to definable manifolds. 

In the setting we are working in, we only have to deal with very simple manifolds, so instead of referring to the theory of definable manifolds, we explain the necessary definitions in the simple context in which they are needed. 

By what we have shown in Section \ref{ss: definable}, we may assume that the group $G$ we are working with is definable. By Theorem \ref{thmFiniteIndex} below\footnote{We note that the proof of Theorem \ref{thmFiniteIndex} is independent of any results proved in Section \ref{S: BaP}.}, after possibly passing to a quotient by a finite subgroup, there exists a  definable group $G_0\le G$ of finite index definably isomorphic to either a subgroup of $(K,+)$, a subgroup of $(K^*, \cdot)$ or to one of the groups $G_{m,a}$ for $a$ of strictly positive valuation (see Definition \ref{D: Gma}). Let $\iota: G_0\to H$ be such an isomorphism for some such group $H$.  Since $G_0$ is infinite, so is $H$. This implies that $H$ contains an open ball. Since $H$ is topological with respect to the affine topology, it follows that $H$ is, in particular, an open subset of $K$ (clearly, this can also be deduced from the classification of 1-dimensional abelian groups definable in ACVF). 

This implies that if $\alpha_1, \dots, \alpha_k$ are representatives of all the cosets of $G_0$ in $G$, then setting $\iota_i:\alpha_i G_0\to H$ to be $x\mapsto \iota(\alpha^{-1}x)$, we obtain a (definable) atlas for $G$, giving also a definable manifold structure to $G^n$ (any $n$). We topologise $G$ via these charts (declaring the chart maps to be homeomorphisms). A definable function $f: G\to G$ can be identified, through these charts, with a finite (disjoint) union of maps  $f_i: H\to H$ and all (local) properties of definable functions in $K$ extend, naturally, to $G$ via this identification. We will, therefore, use freely such notions as continuity, openness (of maps), differentiability, and analyticity (when it makes sense) for definable functions from $G$ to $G$. We remark that while the notion of the derivative of a function $f: G\to G$ at a point $x$ may depend on the choice of the charts $\iota_i$, notions such ``$f_1$ and $f_2$ have the same derivative at $x$'' are independent of such choices, and are -- ultimately -- all we need. For our purposes, it will suffice, however, to fix once and for all the charts $\iota_i$ and  set $\hat f(x):=\iota_i(f(\iota_j^{-1}(y)))$ where $y=\iota_j(x)$ and $f(x)\in \alpha_iG_0$. 

With this simple machinery, the definitions and results of Section \ref{ss; complete models} go through, unaltered. This is obvious for local results (as will be all results in Section \ref{BITIA}). For the only global results that we need (Lemma \ref{L: t-an} and its analogue Lemma \ref{L: generically locally analytic}) the transfer is obvious via finite atlases (as is our case). 

Note, in particular, that $H$ (the target of the definable isomorphism $\iota$) is -- in a complete model -- an analytic group (because the group operation coincides locally with one of addition, multiplication or twisted multiplication).  In order not to get into the details of discussing analyticity in two variables, for our purposes it suffices that $H$-translations are analytic, and that if two functions are analytic then so are their $H$-product and $H$-inverses. By what we have just said, the same remains true in $G$.  

\textbf{From now on we fix once and for all an atlas for $G$, as described above, and we understand all geometric notions in $G$ in terms of this manifold structure.} It may be worth pointing out, also, that the above arguments show that $\CG$ is a geometric structure (so, in particular, admits a natural notion of dimension, genericity etc.). To keep the notation simple, all local geometric arguments, henceforth, will be written in $K$, even when understood in $G$.

\section{Groups locally equivalent to $(K,+)$ or to $(K^*,\cdot)$}\label{EOTS}
In this section we study groups definable in ACVF which are locally equivalent to $(K,+)$ or to $(K, \cdot)$. This allows us, ultimately, to reduce Theorem \ref{main} to two specific cases that  will be treated separately. This is the only part of the paper that does not seem to extend to analytic expansions  of ACVF, since it builds on the fact that (by elimination of quantifiers) the model theoretic algebraic closure, $\acl(\cdot)$, coincides with the field theoretic algebraic closure. This allows us, as in \cite[\S 3]{HRPIGroups} to find a local finite-to-one homomorphism of a definable one dimensional  group $G$ with an algebraic group. \\

As already mentioned, the results of this section do not depend on anything proved in Section \ref{S: BaP}.

\subsection{On one dimensional algebraic groups.}  Throughout this subsection, we are working in the reduct of $\CK$ to the pure field language. 

Below, algebraic groups need not be connected, i.e., $H$ is an algebraic group if it is a constructible set together with an algebraic group operation. We denote by $H^0$ the Zaritsky-connected component of the identity (which is a subgroup). 
We often refer to $H^0$ as the connected component of $H$ and we say that $H$ is connected if $H=H^0$.

The following is well known.

\begin{fact}\label{finite groups}
A connected one dimensional algebraic group is  either an elliptic curve, or ACF-isomorphic to either $\mathbb G_a$ or $\mathbb G_m$.
\end{fact}

The following result (also well known) follows from following two easy facts concerning  algebraically closed fields of characteristic $p>0$: For any $n\in \mathbb N$, the map  $x\mapsto  x^n$ is a surjective endomorphism of $\mathbb G_m$ whose kernel is a cyclic subgroup of order $n$;  and the map $x\mapsto x^{p^n}-x$ is a surjective endomorphism of $\mathbb G_a$  whose kernel is definably isomorphic to any subgroup of $\mathbb G_a$ of order $p^n$.  

\begin{fact}\label{finite quotients}
  If $G$ is either $\mathbb G_a$ or $\mathbb G_m$ and $L$ is a finite subgroup, then $G/L$ is definably isomorphic to $G$.
\end{fact}

For the next proposition, we remind that two groups $H_1, H_2$ are \textit{(definably) isogenous} if there is a (definable) subgroup $C\sub H_1\times H_2$ whose projection onto both coordinates are surjective with finite kernel.

\begin{proposition}\label{equivalent groups}
Let $(H, \oplus)$ and $(H', \otimes)$ be $F$-locally equivalent 1-dimensional algebraic groups over some fixed algebraically closed field $F$.  
Then: 

\begin{enumerate}
\item 
$H^0$ and $(H')^0$ are definably isogenous. 

\item $H^0$ and $(H')^0$ are isomorphic, or they are both elliptic curves.  
\end{enumerate}
\end{proposition}

\begin{proof}
Let $(a,b,a\oplus b)$ and $(a', b', a'\otimes b')$ be equivalent group triples over $F$ (so the equivalence is witnessed by an algebraic correspondence over $F$) as provided by local equivalence.

Consider the group $H\times H'$ and the type $p:=\tp(a,a'/F)$. Recall that the stabilizer of $p$ is defined as
\[\mathrm{Stab}(p):=\{g: p=gp\}\]
where $gp:=\{\phi(z)\in \mathcal L_F\mid \phi(gz)\in p\}$.  Because $H, H'$ are $\omega$-stable $\mathrm{Stab}(p)$ (or, equivalently, the stabilizer of the unique global extension of $p$ if one is following \cite{marker}) is a definable subgroup of $H\times H'$ (see, e.g., \cite[Theorem 7.1.10 ]{marker}). 

\begin{claim}
 The dimension of  $\mathrm{Stab}(p)$ is equal to the dimension of $p$ and the image of the projection on each coordinate is infinite.   
\end{claim}
\begin{claimproof}
Let $c=a\oplus b$ and $c'=a'\otimes b'$, and let $a_1,a_1', b_1, b_1'$ be realizations of $\tp(a,a',b,b'/cc')$ independent of $a,a',b,b'$ over $cc'$. 

It follows that $a_1\oplus b_1\oplus (b)^{-1}=a$ and $a_1'\oplus b_1'\otimes (b')^{-1}=a'$ so that $(b_1\oplus (b)^{-1},  b_1'\otimes (b')^{-1})\models \mathrm{Stab}(p)$. Since $b_1\not\in \acl(b)$ and $b_1'\not\in \acl(b')$ this shows that the projection of $\mathrm{Stab}(p)$ on each coordinate is infinite.   In fact,
\begin{flalign*}
\dim(p) & =  \dim\left(\left(b_1,b_1'\right)/cc'\right)    \\
 & =  \dim\left(\left(b_1,b_1'\right)/cc'aa'bb'\right)    \\
 & =  \dim\left(\left(b_1\oplus b^{-1}, b_1'\otimes \left(b'\right)^{-1}\right)/cc'aa'bb'\right)    \\
 & \geq  \dim\left(\left(b_1\oplus b^{-1}, b_1'\otimes \left(b'\right)^{-1}\right)\right)    \\
 & =  \dim\left(\mathrm{Stab}\left(p\right)\right).    \\
\end{flalign*}

A similar proof (\cite{marker}[Theorem 7.1.10]) shows that the dimension of $\mathrm{Stab}(p)$ is smaller than or equal to the dimension or $p$. 
\end{claimproof}

$\mathrm{Stab}(p)=\Delta$ is a subgroup of $H\times H'$ with infinite projection on each coordinate. Thus, $\Delta$ is an isogeny between $\pi_1(\Delta)$ and $\pi_2(\Delta)$ which are infinite subgroups of $H$, $H'$ so they must have finite index.

By \cite[7.1.12]{marker} $\Delta$ is connected, so the projections $\pi_1(\Delta)$ and $\pi_2(\Delta)$ are infinite connected subgroups of $H$ and $H'$ hence must  equal  $H^0$ and $(H^0)'$ respectively, and $\Delta$ defines an isogeny between $H^0$ and $(H^0)'$.

This proves the first part of the statement.

\medskip

To prove the last statement, for simplicity of notation assume that $H=H^0$ and $H'=(H')^0$, and consider 

\[
\ker_{\Delta}(\pi_1):=\{x'\in H'\mid (e_H,x')\in \Delta\}\] and 
\[
\ker_{\Delta}(\pi_2):=\{x\in H\mid (x,e_{H'})\in \Delta\}\] then $\ker_{\Delta}(\pi_2)\times \ker_{\Delta}(\pi_1)$ is a subgroup of $\Delta$ and 
\[\Delta/\left(\ker_{\Delta}\left(\pi_2\right)\times \ker_{\Delta}\left(\pi_1\right)\right)\] is the graph of an isomorphism between $H/\ker_{\Delta}(\pi_2)$ and $H'/\ker_{\Delta}(\pi_1)$.

These are all connected one dimensional algebraic groups, which we know are either $\mathbb G_a$, $\mathbb G_m$, or an elliptic curve. So by Fact \ref{finite quotients}, we only need to prove that the quotient of an elliptic curve $E$ by a finite subgroup $\Lambda$ is not isomorphic to  $\mathbb G_a$ nor $\mathbb G_m$ and that $\mathbb G_a$ and $\mathbb G_m$ are not (abstractly) isomorphic. This follows by considering the $m$-torsion points for any $m$ relatively prime to the characteristic of the field and to the order of $\Lambda$: By  \cite[Corollary III,6.4]{Silverman} the $m$-torsion points in an elliptic curve (in particular of $E$ and by choice of $m$ of $E/\Lambda$) are isomorphic to $\mathbb Z/m\mathbb Z\times \mathbb Z/m\mathbb Z$ whereas there are no non-trivial $m$-torsion points in  $\mathbb G_a$ and the $m$-torsion points in $\mathbb G_m$ are isomorphic to $\mathbb Z/m \mathbb Z $.
\end{proof}

\begin{remark}\label{generic equivalent triples}
Assume that  $(H, \oplus)$ and $(H', \otimes)$  are connected locally equivalent one dimensional algebraic groups, witnessed by  $(a,b,a\oplus b)$ and $(a', b', a'\otimes b')$. Let $\theta(x,y)$, $\mu(x,y)$ and $\nu(x,y)$ be field formulas implying the inter-algebraicity of $a$ and $a'$, $b$ and $b'$ and $a\oplus b$ and $a'\otimes b'$, respectively.

The formula (in the language of rings)
\[
(\exists z,w\in H')( \theta(x,z)\wedge \mu(y,w) \wedge \nu\left(x\oplus y, z\otimes w\right))
\]
is  satisfied by $(a,b, a\oplus b)$ and implies the existence of a group triple in $H'$ equivalent to $(a,b, a\oplus b)$. Since $H$ is strongly minimal (with its ACF induced structure)  and $a,b$ were generic independent, this formula is satisfied by any algebraically independent elements in $H$. 

So given any $\alpha, \beta$ generic independent elements in $H$, there are $\alpha', \beta'$ in $H'$ such that $(\alpha, \beta, \alpha\oplus \beta)$ and $(\alpha', \beta', \alpha'\otimes \beta')$ are locally equivalent group triples. \end{remark}

\subsection{Locally affine groups}

In this section, we classify definable groups locally equivalent to $(K,+)$ or to $(K^*, \cdot)$. For the statement we need: 
\begin{definition}\label{D: Gma}
For $a\in \mathbb G_m$ with positive valuation, the group $\mathbb G_{m,a}$  is the structure with universe 
\[\{x\mid 0\leq v(x) < v(a)\}\] and group operation 
\[
x\odot y:=\begin{dcases*}
x\cdot y, & v(x)+v(y)<v(a) \\
x\cdot y\cdot a^{-1} 
   & \text{Otherwise}.
\end{dcases*}
\]

\end{definition}

In those terms, we show: 

\begin{thm}\label{thmFiniteIndex}
Let $(G,\oplus)$ be an abelian $\CK$-definable group. 

\begin{itemize}
\item If  $G$ is locally equivalent to $\mathbb G_a$, then there is a definable subgroup $G_0\le G$ of finite index and a finite subgroup $F$ of $G_0$ such that $G_0/F$ is definably isomorphic to a subgroup of the additive group.

\item  If $G$ is locally equivalent to $\mathbb G_m$ then there is a definable subgroup $G_0\le G$ of finite index and a subgroup $F$ such that $G_0/F$ is isomorphic to either a subgroup of the multiplicative group $\mathbb G_m$, or to $\mathbb G_{m,x}$ for some $x\in \mathbb G_m$.
\end{itemize}

\end{thm}

The rest of this section is dedicated to the proof of the theorem. First, we need some definitions for the statement of the next fact.

\begin{definition}
Let $a\in K$ with $v(a)>0$. Then $o(a)$ is the type definable multiplicative subgroup, consisting of the set of elements $x$ such that $|nv(x)|<v(a)$ for all $n\in \mathbb N$.
\end{definition}

The next Fact follows from Propositions 36 and 46 of \cite{acostaACVF}.

\begin{fact}\label{fact Subgroups Multiplicative Acosta}
    \begin{itemize}
       \item Any type definable subgroup of $\mathbb G_a$ is the intersection of definable subgroups.
        \item A type definable subgroup of $\mathbb G_m$ is either an intersection of definable subgroups or a small intersection of groups of the form $o(a)$. 
     \end{itemize}
\end{fact}

We  need the following result, which is an easy application of   \cite[Theorem 2.19]{montenegroOnshuus}. We give a more general formulation than what is needed in the sequel, since it may be applicable in other contexts of interest. Since we only need the conclusion of the proposition, and the proof involves some rather heavy terminology that is not need anywhere else in the paper, we refer the reader to  \cite{montenegroOnshuus} for the definitions.

\begin{proposition}\label{factAlgebraicGrouChunkMontenegro}
    Let $T$ be a dependent (countable) theory expanding the theory of fields. Assume that $T$ is algebraically bounded and that any model of $T$ is definably closed in its algebraic closure. Let $G$ be a definably amenable group, definable in some $\aleph_0$-saturated model $M$ of $T$. Then there is an algebraic (algebraically connected) group $H$, a type definable subgroup $D$ of $G$ of bounded index, and a definable finite-to-one group homomorphism from $D(M)$ to $H(M)$.
\end{proposition}

\begin{proof}
In a dependent theory definably amenable groups admit bi-f-generic types\footnote{A type $p$ over $M$ is bi-f-generic over $M$ if given  $\phi(x)\in p(x)$, any translate (left or right) of $\phi(x)$ by elements of $G(M)$ does not fork over $M$} (by \cite[Proposition 5.9]{HrPiNIP}  there are f-generic types and by \cite[Lemma 3.6]{montenegroOnshuus} this implies the existence of bi-f-generic types), and the ideal $\mu_G$ of formulas not extendable to bi-f-generic types is an $M$-invariant, $S1$-ideal (\cite[Corollary 3.5]{ChSi} for ideal, \cite[Lemma 3.14]{montenegroOnshuus} for S1), stable under left and right multiplication. Now, by \cite[Proposition 5.44]{Si} any type over a model of a dependent theory satisfies condition $(F)$  of \cite[Theorem 2.19]{montenegroOnshuus}   \footnote{Condition (F) for a type $p$ over a model $M$, as stated in  \cite[Theorem 2.15]{montenegroOnshuus} is weaker than strict non forking of $p$ over $M$ (see  \cite[Definition 5.40, Remark 5.41]{Si}).  By \cite[Example 5.43]{Si} every type over a model $M$ of a dependent theory is strictly non forking over $M$. So condition (F) holds for any type over a model in all dependent theories.}. So all the hypothesis of Theorem 2.19 in \cite{montenegroOnshuus} are satisfied, and we can apply it to get a $\mu_G$-generic type definable subgroup $D$ of $G$ with a finite to one group homomorphism from $D(M)$ to $H(M)$. Since bi-f-generic type-definable groups (in fact, all f-generic type definable subgroups must have bounded index by \cite[Theorem 1.2]{ChSi}), the conclusion holds.

As for connectedness of $H$, if $H$ is not connected,  we can take  the inverse image $D^0$ of the connected component  $H^0$ of $H$ in $D$. As $H^0$ is definable,  $D^0$ is type definable and  $H^0$ has finite index in $H$ which implies that $D^0$ must have finite index in $D$. It follows that $D^0(M)$ has bounded index, and $D^0, H^0$ satisfy the conclusion of the proposition with $H^0$ connected, as required.
\end{proof}

We need the following easy observation:

\begin{remark}
\label{reducing equivalence to connected}\label{reducing equivalence to subgroups}
Let $(H, \oplus)$ and $(H', \otimes)$ be locally equivalent 
definable groups witnessed by $(a,b,a\oplus b)$ and $(a', b', a'\otimes b')$. Let $D$ be a type definable subgroup of $H$ of bounded index. Then there are $a_D$ and $b_D$ in $D$ such that $(a_D,b_D,a_D\oplus b_D)$ and $(a', b', a'\otimes b')$ are equivalent group triples. 
\end{remark}

\begin{proof}
Let $(a,b,a\oplus b)$ and $(a', b', a'\otimes b')$ be as in the statement. Since $D$ has bounded index it must have the same dimension as $G$, so we can find elements $a_1,b_1$ in the same $D$-class as $a,b$ respectively, and such  $a,b$ are algebraically independent of $a_1,b_1$ over all the parameters involved in the definition of $H,H'$ and $D$. 

Then the triple 
\[\left(a_1\oplus a^{-1}, b_1\oplus b^{-1}, \left(\left(a_1\oplus a^{-1}\right)\right)\oplus \left(b_1\oplus a^{-1}\right)\right)\] is equivalent to $(a', b', a'\otimes b')$ over $a_1b_1$, as required.\end{proof}

We are ready to proceed to the proof of Theorem \ref{thmFiniteIndex}.

\begin{proof}[Proof of Theorem \ref{thmFiniteIndex}]

That ACVF is dp-minimal (so, in particular, dependent) is proved in \cite[\S 4]{DoGoLi}, algebraic boundedness follows from quantifier elimination (as discussed following the proof of Proposition \ref{decompositionProp}), and  since $G$ is  abelian, it is amenable (see \cite{wagon}) so in particular definably amenable. Thus, we may apply Proposition \ref{factAlgebraicGrouChunkMontenegro} to obtain an algebraic (algebraically connected) group $(H, \oplus_H)$, a $\CK$-type definable subgroup $D$ of $G$ of bounded index and a $\CK$-definable finite-to-one group homomorphism $\phi:D\to H$. The statement of Theorem \ref{thmFiniteIndex} allows replacing $G$ with a quotient by a finite subgroup, and  as $\ker(\phi)$ is finite we may  identify $G$ with $G/\ker(\phi)$ and assume that $\ker(\phi)$ is trivial (and that $\phi$ is an injection).

Let $(G,\oplus)$, $\alpha,\beta\in G$ and $a,b\in K$ witness the local equivalence of  $G$ with either $\mathbb G_a$ or $\mathbb G_m$.

\begin{lemma}
If $(\alpha,\beta,\alpha \oplus \beta)$ and $(a,b,a+b)$ are equivalent group triples, 
then 
$H$ is $\CK$-definably isomorphic to the additive group.

Similarly, if $(\alpha,\beta,\alpha \oplus \beta)$ and $(a,b,a\cdot b)$ are equivalent group triples, 
then 
$H$ is $\CK$-definably isomorphic to the multiplicative group.
\end{lemma}

\begin{proof}
The proofs are identical, so we prove the additive case.

By Remark \ref{reducing equivalence to subgroups} there are $\alpha_D, \beta_D \in D$ such that $(\alpha_D,\beta_D,\alpha_D\oplus\beta_D)$ is $\CK$-inter algebraic with $(a',b',a'+b')$. Being equivalent group triples is of course preserved under definable finite to one group correspondences, so if $a_H, b_H$ are the images of $\alpha_D, \beta_D$ in $H$, we have that 
$(a',b',a'+b')$ and $(a_H,b_H,a_H+_H b_H)$ are equivalent group triples. 

As $H$ and the additive group are both algebraically connected, Proposition \ref{equivalent groups} gives, using the fact that ACVF is algebraically bounded, that they are isomorphic, as claimed. 
    
\end{proof}

We identify $H$ with its image under the definable isomorphism with either $\mathbb G_a$ or $\mathbb G_m$, so that we have $\phi:D(M)\to H(M)$ an injective group homomorphism such that $H(M)$ is either the additive or the multiplicative group.

First, we deal with the case where $\phi(D)$ is an intersection of definable subgroups of $H$:
\begin{claim} \label{claimIntersectionOfGroups}
    If $\phi(D)$ is an intersection of definable subgroups of $H$, then there are $C\leq G$, a definable subgroup with $D\subseteq C$, and $\phi':C\to H$, an injective definable group homomorphism extending $\phi$.
\end{claim}

\begin{claimproof}
    We follow the proof of Proposition 4.7 of \cite{acosta}.
By compactness there are definable sets $U_1$ and $U_0$ with $U_1=U_1^{-1}$, $D\subseteq U_1\subseteq U_0\subseteq G$ and $U_1 \oplus U_1\subseteq U_0$ such that the formula defining $\phi$ also defines a function $\phi:U_0\to  H$ that is injective and respects the group operation.

By compactness and the fact that $\phi(D)$ is an intersection of definable groups, there is some definable group $B\subseteq H$ such that $\phi(D)\subseteq B\subseteq \phi(U_1)$.

Take $$C:=\phi^{-1}(B)\cap U_1.$$
Clearly $C$ is definable, contains $D$ and $\phi$ is defined on $D$ and respects the group operation. So we only have to prove that $C$ is a subgroup of $G$. If $a\in C$ then $a^{-1}\in U_1$ and $\phi(a^{-1})=-\phi(a)\in B$ so $a^{-1}\in C$. Now let $a,b\in C$ so $a\oplus b \in U_0$, moreover $\phi(a\oplus b)=\phi(a)+_H\phi(b)$. As $B$ is a subgroup containing $\phi(a)$ and $\phi(b)$ it follows that $\phi(a\oplus b)\in B$. In particular, as $B\subseteq \phi(U_1)$, there is some $c\in U_1$ such that $\phi(a\oplus b)=\phi(c)$. As the restriction of $\phi$ to $U_1$ is injective, it implies that $a\oplus  b= c$ so $a\oplus b\in C$ and we conclude. \end{claimproof}

Thus, if $\phi(D)$ is a small intersection of subgroups and $C$ is as provided by Claim \ref{claimIntersectionOfGroups}, $C$ must also have bounded index and since it is definable, it must have finite index in $G$. $C$ is definably isomorphic to the definable subgroup $\phi(C)$ of $H$ so that the theorem holds.\\

So we are only missing the case where $\phi(D)$ is not the intersection of definable groups. By Fact \ref{fact Subgroups Multiplicative Acosta} this implies that $H$ is the multiplicative group and $\phi(D)$ is a small intersection $\bigcap\limits_{i\in I} o(a_i)$ where we remind that 
\[
o(a_i)=\bigcap_{n\in \mathbb N} \{x : |nv(x)|<v(a_i)\}
\]
for some $a_i$ of positive valuation.

 Let $\psi$ be the inverse of $\phi$, and for any rational $q$ let $X_q(a_i)=\{x : |v(x)|<qv(a_i)\}$. By compactness, there are  $N\in \mathbb N$,  some $i$ such that $X_{1/N}(a_i)\supseteq \phi(D)$ and a map extending $\psi$ from $X_{1/N}(a_i)$ into $G$ such that $\psi$ is injective and such that it is a \emph{local group homomorphism in $X_{1/N}(a_i)$}: whenever $x,y,x\cdot y\in X_{1/N}(a_i)$ we have $\psi(x\cdot y)=\psi(x)\oplus \psi(y)$.

Denote $a:=a_i$ and $X:=X_{1/N}(a)$.

Notice that, since we are working in the multiplicative group, for any $\delta$ we have that $\left(X_{\delta}(a)\right)\cdot \left(X_{\delta}(a)\right)= X_{2\delta}(a)$. We will use this to inductively prove that for $k\in \mathbb N$, the natural extension of $\psi$ to $X_{2^k/N}(a)$
defined as  $\psi(x\cdot y)=\psi(x)\oplus \psi(y)$ is a well-defined local group homomorphism in $X_{2^k/N}(a)$:

\begin{claim}
    Assume that we have a local group homomorphism $\psi$ from $X_\delta(a)$ to $G$. Then the map
    from $X_{2\delta}(a)$ defined as $\psi(x\cdot y):=\psi(x)\oplus \psi(y)$ is  well-defined local group homomorphism in $X_{2\delta}(a)$.
\end{claim}

\begin{claimproof}
We first prove that the map is well-defined. Suppose that $x\cdot y=z\cdot w$ for $x,y,z,w\in X_\delta(a)$.

Choose $x_1,y_1,z_1,w_1$ to be square roots of $x,y,z,w$ respectively, such that $x_1y_1=z_1w_1$. By hypothesis $x,y,z,w\in X_\delta(a)$ so that by hypothesis and construction that $\psi(x)=\psi(x_1\cdot x_1)=\psi(x_1)\oplus \psi(x_1)$, and similarly $\psi(y)=\psi(y_1)\oplus \psi(y_1)$, $\psi(z)=\psi(z_1)\oplus \psi(z_1)$, and $\psi(w)=\psi(w_1)\oplus \psi(w_1)$. On the other hand, by construction $x_1\cdot y_1$ and $w_1\cdot z_1$ are also in $X_\delta(a)$ so that again by hypothesis $\psi(x_1\cdot y_1)=\psi(x_1)\oplus \psi(y_1)$ and $\psi(z_1\cdot w_1)=\psi(z_1)\oplus \psi(z_1)$. Combining all this and using that $G$ is abelian, we get that:

$\psi(x)\oplus \psi (y)=\psi(x_1)\oplus\psi(x_1)\oplus\psi(y_1)\oplus\psi(y_1)=\psi(x_1y_1)\oplus \psi(x_1y_1)=\psi(z_1w_1)\oplus\psi(z_1w_1)=\psi(z_1)\oplus \psi(z_1)\oplus \psi(w_1)\oplus\psi(w_1)=\psi(z)\oplus\psi(w)$ as required.

\medskip

Assume now that $x\cdot y=z$ and that all elements are in $X_{2\delta}(a)$. As before, let $x_1,y_1$ be square roots of $x$ and $y$ respectively. We get that  $x_1,y_1, x_1\cdot y_1$ are all elements of $X_\delta(a)$, and that by definition (and what we have just shown) $\psi(x)=(\psi(x_1))\oplus (\psi(x_1))$, $\psi(y)=(\psi(y_1))\oplus (\psi(y_1))$ and $\psi(z)=(\psi(x_1\cdot y_1))\oplus (\psi(x_1\cdot y_1))$. Commutativity and the assumption that $\psi$ is a local homomorphism on $X_\delta(a)$ proves that $\psi(x)\oplus \psi(y)=\psi(z)$, as required.
\end{claimproof}

Recall that $D$ has bounded index in $G$, which implies that finitely many translates of $\psi(X)$ cover $G$. This implies that for some $k$, the set
\[\underbrace{\psi\left(X\right)\oplus \dots \oplus \psi\left(X\right)}_{k\text{ times}}\] is already the subgroup $G'$ of $G$ generated by $\psi\left(X\right)$.

This implies that  \[\psi(X_{k/N}\left(a\right))=\psi\left(X_{(k+1)/N}\left(a\right)\right),\] so there must be some element in $X_{(k+1)/N}(a)$ different from 1 that is in the kernel of $\psi$. Because the restriction of $\psi$ to $X_{1/N}(a)$ is injective, the valuation of any such element must be grater than or equal to $v(a)/N$. Let $l$ be the smallest integer such that $X_{l/N}(a)$ contains an element in the kernel of $\psi$ different from 1.

\begin{claim}\label{claimUniqueInKernel}
There is a unique element in $X_{l/N}(a)$ different from the identity in the kernel of $\psi$.
\end{claim}

\begin{claimproof}
This is a standard lattice argument, but we include it for completeness. We know that injectivity of $\psi$ in $X_{1/N}(a)$ implies $l>1$. 

Let $x,y\in  X_{l/N}(a)\cap \ker(\psi)$ with $x,y\neq 1$ so that $x,y\not \in  X_{{l-1}/N}(a)$ by minimality of $l$. Assume $v(x)\geq v(y)$, so that 
\[\frac{l-1}{N}v(a)\leq v(y)\leq v(x)< \frac{l}{N}v(a),\]
so \[v(xy^{-1})=v(x)-v(y)< \frac{l}{N}v(a)-\frac{l-1}{N}v(a)=\frac{1}{N}v(a).\]

It follows that $x\cdot y^{-1}\in X_{1/N}(a)\cap \ker(\psi)$ which implies $x=y$ by injectivity of $\psi$ in $X_{1/N}(a)$. 
\end{claimproof}

If $b\in \ker(\psi)\cap X_{l/N}(a)$ is as provided by Claim \ref{claimUniqueInKernel}, then the universe $\{x\mid 0\leq v(x)\leq v(b)\}$ of $G_{m,b}$ is contained in $X_{l/N}(a)$. The reader can check that the restriction of the map $\psi$ to $G_{m,b}$ is a group isomorphism from $G_{m,b}$ to $G'$.  
\end{proof}

\section{Basic intersection theory in ACVF}\label{BITIA}
		
    Our goal in this section is to develop the tools allowing us to detect, definably in the reduct, tangency of definable curves. The standard ideology is that if two curves are tangent at a point $x_0$ then a small perturbation cannot lose any intersection points and that, in fact, near $x_0$ some intersection points must be gained. Thus, tangent curves should have fewer than the generic number of intersection points (not counting multiplicities), a property definably detectable in the relic.  Below we  formulate this idea precisely (and prove it) for suitably chosen families of curves. 
		
	To keep the notation cleaner we develop the theory in $K^2$, but as explained in Section \ref{ss: the setting}, everything transfers smoothly to $G^2$ via the simple definable atlas introduced there.

	Since ACVF$^{eq}$ has uniform finiteness (\cite[Theorem 3.1]{JohCminimalexist}) if $(G,\oplus)$ is the underlying group of $\CG$, there is a $\0$-definable family of infinite commutative groups, of which $G$ is a generic member. Moreover, by Theorem \ref{thmFiniteIndex} we may require that either all groups in the family have a  subgroup of index $n$ (some fixed $n$) definably isomorphic (via a specific isomorphism) to a definable subgroup of $(K,+)$, or all are isomorphic to  a subgroup of $(K^*, \cdot)$ or all groups in the family are definably isomorphic (via a specific isomorphism) to  a group $\mathbb G_{m,a}$ for some $a$ of positive valuation. In particular, if $\{G_s\}_{s\in S}$ is such a definable family of curves then for all $s\in S(\Kk)$ the groups $G_s(\Kk)$ in the family are analytic (since we have not properly defined this notion, let us remark that for our needs it suffices that $G_s$-translations are analytic). For that reason, as long as we are proving elementary statements, it will suffice to prove them for definable analytic groups $G$, whose universe can be identified with a subset of $K$: 
		
		\textbf{From this point on, we assume that all groups considered are topological  in the affine (i.e., induced) topology, and if definable over $\Kk$ we assume further they are  analytic.}
		
		When working in a complete model, the following is a key definition: 
		\begin{definition}
			Let $\Kk$ be a complete model. A definable curve $C$ in $\mathbb K^2$ is \emph{$t$-analytic at $c\in C$} if there is a definable $t$-analytic function $f_c$ whose graph is contained in $C$ and contains $c$. The curve $C$ is \emph{locally $t$-analytic }if it is $t$-analytic at every point. 
		\end{definition}
	
		Lemma \ref{L: generically locally analytic}  shows that every definable curve is $t$-analytic at all but finitely many points. By the same results, a curve is $t$-analytic at $x$ but (possibly) not analytic at $x$ only if the function $f$ witnessing $t$-analyticity at $x$ is of the form $\Fr^n \circ g$ for some analytic $g$ and $n<0$.  Note also that, in the notation of the above definition, we do not require that $C$ coincides with the graph of $f_c$ on any neighbourhood. In particular, a curve may be $t$-analytic at a branch point.

        In the applications, we need a first order variant of local $t$-analyticity. For technical reasons, it is useful to introduce a slightly stronger variant: 
        \begin{definition}\label{D: strongly t-an}
            A definable plane curve $C$ is strongly $t$-analytic if for every point $(x_0,y_0)\in C$ the following hold: 
            \begin{enumerate}
                \item There exists a unique irreducible component  $Z\ni (x_0,y_0)$ of the Zariski closure of $C$. 
                \item If $Z$ is the zero set of an irreducible polynomial $H(x,y^{p^n})$ and $n$ is maximal such, then $\frac{\partial H}{\partial y}(x_0,y_0^{p^n})\neq 0$ (and analogously if $Z$ is the zero set of $H(x^{p^n},y)$). 
                \item Under the assumptions of (2) above, there exists a definable function $f: B\to K$ for some ball $B\ni x_0$ such that the graph of $\Fr^{-n}\circ f$ is contained in $C$ near $(x_0,y_0)$ and $f$ is smooth with a smooth local inverse at $y_0^{p^n}$ (and analogously if $Z$ is the zero set of $H(x^{p^n},y)$). 
            \end{enumerate}
            A point $c\in C$ is $t$-smooth if conditions (2) and (3) above hold in $c$. 
        \end{definition}

        It is easy to see that, in a complete model, strong $t$-analyticity implies local $t$-analyticity. Indeed, if the zero-set of $H(x,y^{p^n})$ is the irreducible component of the Zariski closure of $C$ containing $(x_0,y_0)$ and $n$ is maximal possible such, then strong $t$-analyticity implies that $\frac{\partial H}{\partial y}(x_0,y_0^{p^n})\neq 0$, so by the Implicit Function Theorem the zero-set of $H$ coincides, near $(x_0,y_0^{p^n})$, with the graph of a (definable) analytic function $f$ so that $\Fr^{-n}\circ f$ coincides with $C$ near $(x_0,y_0)$. In particular, $C$ is $t$-analytic at $(x_0,y_0)$. The same argument shows that a curve is $t$-analytic at $t$-smooth points. 
        
        We also point out that strong $t$-analyticity is definable, uniformly, in definable families. For this, it suffices to note that the Zariski closure of any definable set $C$ is definable over any set of parameters over which $C$ is defined. Similarly, the statement ``$(x_0,y_0)$ belongs to a unique irreducible component of the Zariski closure of $C$'' is defined over the same set of parameters as $C$.  By the same argument we also get that to any curve $C$ there is a finite set $S$, defined over the same parameters as $C$, such that $C\setminus S$ is strongly $t$-analytic. 
	
		It will be convenient to introduce an \emph{ad hoc} definition: 
		\begin{definition}\label{D: Cgen}
			Let $(G,\oplus)$ be a definable 1-dimensional group and $C\sub G^2$ a definable  curve. 
			\begin{enumerate}
				\item A definable family of plane curves $\CY$ is \emph{ 1-$C$-generated} if it is of one of the following forms: $\{t_a(C): a\in C\}$, $\{t_a(C): a\in G^2\}$, $\{C\ominus t_a(C): a\in C\}$ or $\{C\ominus t_a(C): a\in G^2\}$. A curve is \emph{1-$C$-generated} if it belongs to a 1-$C$-generated family.

    \item A family of curves is \emph{$n$-$C$-generated} if it is either 1-$D$-generated for some $(n-1)$-generated curve $D$, or obtained as either the composition or the functional (group) sum of two $(n-1)$-$C$-generated families. As before, a curve is \emph{$n$-$C$-generated} if it belongs to a $n$-$C$-generated family.
			\end{enumerate} 
			A family of plane curves is $C$-generated if it is $n$-$C$-generated for some $n$, and a curve is $C$-generated if it is $n$-$C$-generated for some $n$. 
		\end{definition}

        We will see below that if $C$ is strongly $t$-analytic then, in a complete model, any curve in a $C$-generated family is locally $t$-analytic. Since, as explained above, to any definable curve we can associate, uniformly, a strongly $t$-analytic curve, any elementary statement we can prove on families of $C$-generated curves for a strongly $t$-analytic curve $C$ in a complete model transfer to any model of ACVF. The advantage is that to prove such statements in a complete model, we can use properties of locally $t$-analytic curves. Another useful feature is that given a relic-definable curve $C$, any $C$-generated family is definable in the relic.  

        So we start with: 
        \begin{lemma}\label{L: uniform t-an}
            Let $C\sub \Kk$ be a strongly $t$-analytic curve definable in a complete model,  $\{Y_s\}_{s\in S}$ a $C$-generated family of plane curves. Then  $Y_s$ is locally $t$-analytic for all $s\in S$. 
        \end{lemma}
        \begin{proof}
            By induction on the generation of $\{Y_s\}_{s\in S}$. As already pointed out above, for $C$ itself this is clear. Our assumption that $G$ is analytic (in complete models) implies that if $D_1, D_2$ are locally $t$-analytic curves then so are  $t_a(D_i)$ and the (functional) sum $D_1\oplus D_2$ and difference $D_1\ominus D_2$. Since the composition of $t$-analytic functions is $t$-analytic, it will suffice to show that the composition of curves $D_1,D_2$ at $t$-analytic points $(a,b)\in D_1$, $(b,c)\in D_2$  contains the graph of the composition of the functions $f_1, f_2$ witnessing local $t$-analyticity of $D_1$, $D_2$ at $(a,b)$, $(b,c)$ respectively.  To that end let $f_i: B_i\to \Kk$ witness $t$-analyticity as above. By \cite[Lemma 2.3]{Bendetto} analytic functions map discs to discs, and since this is true also of $\Fr$ and $\Fr^{-1}$ it follows that $f_1$ maps $B_1$ to a disc containing $b$. Reducing $B_1$, if needed, we may assume that $f(B_1)\sub B_2$, and the conclusion follows. 
        \end{proof}

		Clearly, given a curve $C$, any $C$-generated family of plane curves is also definable (and if $C$ is relic definable then the so is any $C$-generated family). In order to formulate first order statements concerning $C$-generated families of curves, we need to work uniformly in generated families, as $C$ ranges over a definable family of plane curves. To do so, it is useful to introduce:  
		\begin{definition}
			A generation scheme is a formula $\phi_C(x,y; s):=\exists x',y'  ((x',y')\in C\land \psi(x',y',x, y, s))$ such that for any plane curve $C$ the formula $\phi_C(x,y,s)$ defines a $C$-generated family of plane curves indexed by $S:=\{s: (\exists^\infty x,y) \phi_C(x,y,s)\}$. 
		\end{definition}

        So a statement ``given a definable curve $C$ and a $C$-generated family...'' translates into a series of first order statements by taking $C$ to be a (generic) member of a $\emptyset$-definable family of curves $\mathcal C$ and, for every generation scheme $\phi_C(x,y,z)$,  we have the statement ``For any $D\in \mathcal C$ the family of plane curves $\phi_D(x,y,z)$...''. To keep the exposition cleaner, we do not dwell further on this point.   
				
		We now start developing the intersection theory needed to carry out the strategy described above. We need some preliminaries: \\
		
		In a complete model $\Kk$, we usually identify $\Gamma$ with an additive subgroup of  $(\RR, +)$. It is sometimes convenient to replace this with a multiplicative subgroup by setting $|x|:=e^{-v(x)}$ for $x\in K^*$ and $|0|:=0$ (note that this is also order reversing). To conform with the notation in the relevant literature, and in order to better capture the analogy with classical analytic geometry, in the present section we use this notation, working with absolute values instead of valuations, and identifying $\Gamma$ with a subgroup of $(\RR^{>0}, \cdot)$. 
		
		\begin{fact}[Theorem 2.11  \cite{Cher2}]\label{F: Cherry}
			Let $\Kk$ be a complete algebraically closed  valued field, $f=\sum\limits_{n=1}^\infty   c_nz ^n$  an analytic function converging in $B_R(0)$. For $r<R$ let $|f|_r:=\sup_n \{|c_n|r^n\}$ and \[K(f,r):=\max\{n\in \mathbb N: |c_n|r^n=|f|_r\}.\] 
			Then the number of zeros of $f$ in $B_r(0)$, counting multiplicities, is $K(f,r)$. 
		\end{fact}
		
		Let us define: 
		\begin{definition}
			Let $\{f_s\}_{s\in S}$ be a family of $t$-analytic functions with common domain $B$. Say that the coefficients of the power series expansion at $b\in B$ are $N$-continuous at $s$ if for every $r\in \Gamma$ there is an open set $V\ni s$ and $n\in \mathbb N$,  such that $\Fr^n\circ f_t$ is analytic in $B$ 
   and $|c_i(s)-c_i(t)|<r$ for all $i\le N$ and $t\in V\cap S$, where $\Fr^n\circ f_t(x)=\sum_i c_i(t)(z-b)^n$. 
		\end{definition}
	
	In those terms, we need the following observation: 
		\begin{lemma}\label{L: Cont KFR}
			If the coefficients of the power series expansion at $b$ of the family $\{f_s\}_{s\in S}$ of analytic functions are $K(f_s,r)$-continuous at $s$, then $K(f_s,r)\le K(f_t,r)$ for all $t\in S$ close enough to $s$.
		\end{lemma}
		\begin{proof}
			Let $N:=K(f_s,r)$. For every $n\le N$ let $0<r_n<|c_n(s)|$ if $c_n(s)\neq 0$. If $c_n(s)=0$ choose some $n_0\le N$ such that $c_{n_0}(s)\neq 0$ and choose $0<r_n$ such that $r_n r^n <|c_{n_0}|r^{n_0}$. Let $r:=\min\limits_{n<N} \{r_n\}$, and let $V_r\ni s$ be a ball as provided by the assumption on the continuity of the  coefficients of the power series expansion with respect to $r$ and $N$. Then for any $t\in S$ and $n\le N$ such that $c_n(s)\neq 0$ we have $|c_n(s)-c_n(t)|<r<|c_n(s)|$, so $|c_n(s)|=|c_n(t)|$. By the choice of $r_n$ in the case $c_n(s)=0$ this implies that $K(f_t,r)\ge K(f_s,r)$.  
		\end{proof}
	
	Consequently, we get 
	\begin{corollary}\label{C: keeping intersection}
		Let $\{f_s\}_{s\in S}$ be a family of non-constant analytic functions with domain $B$. If the coefficients of the power series expansion at $b$ are $2$-continuous at $s$ and $f_s(b)=0$,  then for any $t$ close enough to $s$ there is some $b_t\in B$ such that $f_t(b_t)=0$. 
	\end{corollary}
	\begin{proof}
		Let $r\in \Gamma$  be such that $b$ is the unique zero of $f$ in $B_r(b)$ (recall that, by the Identity Theorem, the zeros of analytic functions are discrete). By Fact \ref{F: Cherry} (applied to $f_s$) we get that $K(f_s,r)\ge 1$. By  Lemma \ref{L: Cont KFR} $K(f_t,r)\ge 1$ for all $t$ close enough to $s$, so by applying Fact \ref{F: Cherry} to every such $f_t$, the conclusion follows.  
	\end{proof}
		
		It is not hard to see that in a complete model, the coefficients of the power series expansion of a definable (locally) analytic function around $0$ are definable (over any set of parameters over which the function is definable). E.g., the coefficients up to order $n$ are the coefficients of the unique polynomial $P(z)$ of degree $n$ such that $\lim\limits_{z\to 0} \frac{f(z)-P(z)}{z^{n+1}}$ exists in $\mathbb K$. 
 
        It follows from the above discussion that if $\{f_s\}_{s\in S}$ is a definable family of locally analytic functions then for every $n$ there is a formula describing, uniformly in $s$,  the parameters of the power series expansion of $f_s$ up to order $n$. In particular, if $s_0\in S$ is generic, then $\{f_s\}_{s\in S}$ is $N$-continuous at $s_0$ for all $N$. For $C$ a strongly $t$-analytic curve and $\{Y_s\}$ a $C$-generated family of curves, we have an even stronger result. We start with:       
    
    \begin{lemma}
            Let $C\sub \Kk$ be a definable locally $t$-analytic curve, and let $\CY:= \{t_a(C): a\in G^2\}$. Then for every $a\in G^2$ and $x=(x_1,x_2)\in t_a(C)$ there exists a ball $B\ni x$,  an open $V\ni a$ and a natural number $n$ such that for all $b\in V$ the curve $t_b(C)\cap B$ contains the graph of a definable function $f_b$ near $x$,  $\Fr^n\circ f_b$ is analytic, and the coefficients of the power series expansion of members of the family $\{\Fr^n \circ f_b\}_{b\in V}$ at $x$ are $N$-continuous at $a$ for all $N$.   
        \end{lemma}
        \begin{proof}

            As $C$ is locally $t$-analytic,  at every point $z$ it contains the graph of a definable function of the form  $\Fr^{-n_z}\circ f_z$ for some analytic function $f_z$. Since, by our standing assumption $G$ is analytic and (at least locally) commuting with $\Fr$, for every $a=(a_1,a_2)\in G^2$ and $z=(z_1,z_2)\in t_a(C)$ there is a neighbourhood of $z$ where $t_a(C)$ contains the graph of the function $\Fr^{-n_{z\ominus a}} \circ f_{z_1\ominus a_1}(x\ominus a_1)\oplus a_2$, which  -- at least locally -- coincides with $\Fr^{-n_{z\ominus a}} \circ (f_{z_1\ominus a_1}(x\ominus a_1)\oplus a_2)$, which is analytic.

            Note that the power $n_z$ of $\Fr$ needed to turn $C\cap B$ into an analytic curve is locally constant, as it depends only on the polynomial whose 0-set is the irreducible component of the Zariski closure of $C$ containing $z$. It is also easy to check (see  \cite[Proposition 1.4.3]{Cher}, Fact \ref{coefficientsAreAnalyticLemmaFact} below) that the coefficients of the power series expansion of $f_z$ at a point $x'\in \mathrm{dom}(f)$  vary analytically in $z'$. Combined with the analyticity of $G$, the conclusion of the lemma follows. 
        \end{proof}

        In applications, we sometimes have to slightly manipulate $C$-generated families to obtain the desired conclusions. It will, therefore, be useful to work in somewhat greater generality: 
        \begin{definition}
            Let $C$ be a strongly $t$-analytic curve. A definable family of plane curves $\{Y_s\}_{s\in s}$  is \emph{nearly $C$-generated} if there is a $C$-generated family $\{Z_s\}$ and a definable curve $Z$ such that: 
            \begin{itemize}
                \item for all $s\in S$ we have  $Y_s\sub Z_s\circ Z$ and the difference is finite,  
                \item either for all $z=(z_1,z_2) \in Z$, if $x$ is not $t$-smooth then for all $s\in S$ and $(z_0,z_1)\in Z_s$ we have $(z_0,z_2)\notin Y_s$, or
                \item $Z^{-1}$ is $C$-generated and for all $z\in Z$ there is a definable invertible function $f$ whose graph is contained in $Z$ near $z$. 
            \end{itemize}  
        \end{definition}

        \begin{remark}\label{R: nearly t is analytic}
              In a complete model, if $C$ is strongly $t$-analytic and $\{Y_s\}_{s\in S}$  is nearly $C$-generated, then $Y_s$ is locally $t$-analytic for all $s\in Y_s$. By the fact that (in the notation of the above definition) the composition of $t$-analytic functions is $t$-analytic, and the fact that, by Lemma \ref{L: uniform t-an},  $Z_s$ is  $t$-analytic for all $s$ it will suffice to verify that $Z$ is locally $t$-analytic at every $t$-smooth point. If $z\in Z$ is $t$-smooth, then $Z$ is $t$-analytic at $z$ by the implicit function theorem. If $Z^{-1}$ is $C$-generated, it is locally $t$-analytic, and if $f:=\Fr^{n}\circ g$ is a function witnessing it at $z=(z_1,z_2)$ with  $g$ analytic and not a $p$-th power,  then $f$ is injective at $z$, if and only if  $g'(z_1)\neq 0$  (e.g., by Fact \ref{F: Cherry}), and by the Inverse Function Theorem $f^{-1}$ is $t$-analytic near $z$, as required. 
        \end{remark}

        It now follows that: 
        \begin{lemma}\label{L: cont. of coefficients}
            If $C\sub \Kk^2$ is a definable locally $t$-analytic curve and $\{Y_s\}_{s\in S}$ a nearly $C$-generated family of plane curves. Then for every $s\in S$ and $x=(x_1,x_2)\in Y_s$ there exists a ball $B\ni x$,  an open $V\ni s$ and a natural number $n$ such that for all $t\in V\cap S$ the curve $Y_t\cap B$ is the graph of a definable function, $f_t$ such that $\Fr^n\circ f_t$ is analytic and the coefficients of the power series expansion of members of the family $\{\Fr^n\circ f_b\}_{b\in V}$ at $x$ are $N$-continuous at $a$ for all $N$.   
        \end{lemma}
        \begin{proof}
            By Remark \ref{R: nearly t is analytic} it will suffice to prove the lemma for $C$-generated families. This is proved  by induction on the generation of $\{Y_s\}$. For $1$-generated curves, the claim follows immediately from the previous lemma, and the analyticity of $G$. It is clear that the functional (group) sum of two families of curves satisfying the inductive hypothesis also satisfies the inductive hypothesis (since $G$ is assumed to be analytic), and similarly for the composition of two such families. 
        \end{proof}

		We now prove our first technical lemma, asserting that small perturbations within $C$-generated families cannot decrease the number of intersection points between two curves: 
		\begin{lemma}
			Let $C$ be a strongly $t$-analytic curve,  $\{C_s\}_{s\in S}$ and $\{D_t\}_{t\in T}$  nearly $C$-generated families of curves. Assume that $s_0\in S$, $t_0\in T$ and $x_0=(a_0,b_0)\in C_{s_0}\cap D_{t_0}$. Then for any ball $B\ni x_0$ there is an open $V\ni t_0$ such that $D_t\cap C_{s_0}\cap B\neq \0$ for all $t\in T\cap V$. 
		\end{lemma}
		\begin{proof}
		As the statement is elementary, we may prove it in a complete model. Because both families are nearly $C$-generated and $C$ is strongly $t$-analytic, all curves in those families are locally $t$-analytic (Remark  \ref{L: uniform t-an} and Remark \ref{R: nearly t is analytic}). So fix $s_0\in S$, a ball $B\ni x_0$ such that $C_{s_0}\cap B$ contains the graph of a $t$-analytic function $\hat g$ with $x_0$ on its graph. Let $\hat g=\Fr^m\circ g$ for some analytic function $g$. 

        Let $B_1\ni x_0$, $V\ni t_0$ and $n\in \mathbb N$ be as provided by Lemma \ref{L: cont. of coefficients} with respect to the family $\{D_t\}_{t\in t}$, the curve $D_{t_0}$ and the point $x_0\in D_{t_0}$.  Let $f_t$ be the family of analytic functions provided by Lemma \ref{L: cont. of coefficients} (such that the graph of $\Fr^{-n}\circ f_t$  is contained in $D_t\cap B$).  
        Replacing $B$ and $B_1$ with $B\cap B_1$ we may assume $B=B_1$. Assume also, without loss of generality, that $n>m$. Then the  family $\{g\ominus \Fr^{n-m}\circ f_t\}_{t\in T}$ is a definable family of analytic functions whose power series coefficients vary continuously with $t$ (to any degree). So we may apply Corollary \ref{C: keeping intersection} to conclude that for all $t$ close enough to $t_0$ the function $g-\Fr^{n-m}\circ f_t$ has a zero in (the projection of) $B$, and since zeros are unaffected by $\Fr$ this means that  the curves $C_{s_0}$ and $D_t$ intersect in $B$, as claimed.   \end{proof}

		An immediate corollary of the previous lemma is: 
		\begin{corollary}\label{C: keeping intersections}
			Let $C,D$ be definable curves with $C$ strongly $t$-analytic, and $D$ nearly $C$-generated. Let  $\{Y_s\}_{s\in S}$ be a nearly $C$-generated family. Then for all  $s\in S$ the number of intersection points  $|D\cap Y_s|$ (\textbf{not counting multiplicities}) is either infinite or a local minimum. 
		\end{corollary}
		
		We now aim to show that in the situation described in the previous corollary, under suitable assumptions, if $Y_s$ and $C$ are \emph{tangent} (in an appropriate sense explained below) at one of their intersection points then $|Y_{s'}\cap C|>|Y_{s}\cap C|$ for infinitely many $s'\in S$. This will imply that \textit{in the relic} $s$ is non-generic in $S$ over a canonical base for $C$, which is what we need. 
		
		Before introducing tangency of curves (within definable families) we need some preparations. 
		
		\begin{definition}\label{D: Analytic curves}
			Let $\mathbb K$ be a complete model.
			The intersection of two curves $C_1, C_2$ at a  point $x:=(x_1,x_2)\in C_1\cap C_2$ analytic in both has multiplicity  at least (greater than) $N$ if $|C_1\cap C_2|<\infty$ there exist a ball $B$ and definable  analytic functions $f_i: B\to C_i$ witnessing analyticity of $C_i$ at $x$ and such that $x_1$ is a zero of $f_1-f_2$ of order at least (greater than) $N$. 	
		\end{definition}

    The idea is that, given a definable family of curves, all locally analytic at a common point $x_0$, two curves are tangent at $x_0$ if their intersection multiplicity at $x_0$ is greater than the generic intersection multiplicity at $x_0$. It will be convenient to be able to define this notion in arbitrary models, so we define: 
    \begin{definition}
		Let $f: B\to K$ be a definable function. The degree $n$-polynomial approximation of $f$ at $x_0$ is the unique polynomial $P(z)$ of degree $n$ such that $\lim\limits_{z\to 0} \frac{f(z)-P(z)}{z^{n+1}}$ exists, if such a polynomial exists. 
	\end{definition}
  
	\begin{definition}
		Let $\CY:=\{Y_s\}_{s\in S}$ be a definable family of plane curves, all passing through a common point $x_0:=(a,b)$. For $n>0$ \emph{the family $\CY$ has infinitely many $n$-slopes at $x_0$} if there exists a definable family $\{f_s\}_{s\in S }$ of definable functions, such that: 
		\begin{enumerate}
			\item For all $s\in S$ the graph of $f_s$ is contained in $C_s$ around $x_0$. 
			\item Each $f_s$ has a polynomial approximation of order $n$. 
			\item If $(a_0(s), \dots, a_n(s))$ are the coefficients of the polynomial approximation of $f_s$, then the function $s\mapsto (a_0(s), \dots, a_{n-1}(s))$ has finite image and the function $s\mapsto a_n(s)$ is continuous and nowhere locally constant. 
		\end{enumerate}
	\end{definition}
		
	\begin{remark}
		In applications, the family $\CY$ in the above definition may be $\CG$-definable. In this case, we say that $\CY$ has infinitely many first order slopes at $x_0$ if there exists $S_0\sub S$ definable and of full rank/dimension in the full structure, such that the subfamily $\CY_0:=\{Y_s\}_{s\in S_0}$ has infinitely many first order slopes at $x_0$. 
	\end{remark}
		
	We can now define the notion of tangency needed below, and prove its definability. Specifically, our goal is to show that if $C$ is a $\CG$-definable plane curve and $\{Y_s\}_{s\in S}$ is a strongly minimal $C$-generated family of curves with infinitely many $n$-slopes at a point $x_0$ then the relation ``$Y_s$ and $Y_t$ have the same $n$-slope at $x_0$'' is detectable in $\CG$ in the sense that for generic such $s$ and $t$  the above implies that $s\in \acl_\CG(t)$. We do this by showing that (for a suitably selected curve $C$) two such curves $Y_s$ and $Y_t$ have fewer (geometric) intersection points than  pairs of independent generic curves do.

	\begin{lemma}\label{L: multiple zeros}
		Let $\Kk$ be a complete model, and $\{Y_s\}_{s\in S}$ a definable family of locally analytic plane curves with infinitely many $n$-slopes at a point $x_0=(a,b)$. Assume, moreover, that the coefficients of the power series expansion of $Y_s$ at $x_0$ vary continuously to any order at $s_0$ (e.g., if $s_0$ is generic). Let $C$ be any curve analytic at $x_0$. If $Y_{s_0}\cap C$ has multiplicity greater than $n$ at $x_0$, then for any ball $B\ni x_0$ there are $s\in S$ arbitrarily close to $s_0$  such that $|Y_s\cap C\cap B|>1$.   
	\end{lemma}
	\begin{proof}
		For ease of notation, we  prove the lemma for $n=1$. The general case is identical. 
		Let $g_{x_0}$ witness the fact that $C$ is analytic at $x_0$, $\{f_s\}_{s\in S}$ witness the analyticity of $Y_s$ at $x_0$. By assumption $f'_{s}(a)=g'_{x_0}(a)$.  We may  assume that $|B\cap Y_{s_0}\cap C|=1$ (i.e,, that $x_0$ is the unique geometric intersection point of $C$ and $Y_{s_0}$ in $B$).  By assumption, $C$ and $Y_{s_0}$ intersect at $x_0$ with multiplicity $2$ or more.  By Fact \ref{F: Cherry} and Lemma \ref{L: Cont KFR} any curve $Y_s$ with $s$ close enough to $s_0$ intersects $C$  in $B$ at least twice, counting multiplicities. On the other hand, since $\CY$ has infinitely many 1-slopes, there exist $s\in S$ arbitrarily close to $s_0$ such that $f'_s(a)\neq g'_{x_0}(a)$ so that the multiplicity of intersection of $C$ and $Y_s$ at $x_0$ is $1$. Combining these two observations, we get that $B\cap C\cap (Y_s\setminus x_0)\neq \0$, as required.  
	\end{proof}
	
	Let us now formulate a first order version of the last lemma that we can apply in any model.

	\begin{lemma}\label{L: multiple zeros intersection increases}
		Let $C$ be strongly $t$-analytic $D$ a nearly $C$-generated curve.  Let $\{Y_t\}_{t\in T}$ be a nearly $C$-generated family of curves, all passing through a point $x_0\in C$. Assume that $\{Y_t\}_{t\in T}$ has infinitely many $n$-slopes at $x_0$  and that $Y_{t_0}\cap C$ has multiplicity greater than $n$ at $x_0$. Then, for any ball $B\ni x_0$, there are $t\in T$ arbitrarily close to $t_0$  such that $|Y_t\cap C\cap B|>1$.   
	\end{lemma}
	\begin{proof}
		This is a first order statement, so it is enough to prove it in a complete model. As already noted, $C, D$ are locally $t$-analytic. By Lemma \ref{L: uniform t-an} and Remark \ref{R: nearly t is analytic} all curves in $\{Y_t\}_{t\in T}$ are locally $t$-analytic. The assumption that $\{Y_t\}_{t\in T}$ has infinitely many $n$-slopes at $x_0$ implies, by $t$-analyticity, that $Y_{t_0}$ is, in fact, analytic at $x_0$. By Lemma \ref{L: uniform t-an} again,  there is a ball $B\ni x_0$ and a ball $V\ni t_0$ such that $Y_t\cap B$ is analytic for every $t\in T\cap V$. So we may apply the previous lemma to get the desired conclusion, using Lemma \ref{L: cont. of coefficients}. 
	\end{proof}	

		We can thus conclude: 
\begin{corollary}\label{C: definability of tangency}
	Let $(G,\oplus)$ be an infinite commutative group definable in some $ \CK\models \mathrm{ACVF}$ and locally equivalent to $(K,+)$ or to $(K^*, \cdot)$. Let $C\sub G^2$ be strongly $t$-analytic and $D$ nearly $C$-generated. Let  $x_0\in D$,  and $\CY:=\{Y_t\}_{t\in T}$  a 1-dimensional nearly $C$-generated family of plane curves with infinitely many $n$-slopes at $x_0$. If $t_0\in T$ is such that  $D\cap Y_{t_0}$ has multiplicity greater than $n$ at $x_0$,  then there are $t\in T$ arbitrarily close to $t_0$ such that $|D\cap Y_t|>|D\cap Y_{t_0}|$.   
\end{corollary}
\begin{proof}
	Apply the previous lemma with Corollary \ref{C: keeping intersections}. 
\end{proof}

		\section{First order slopes and a first field configuration}\label{FOSAAFFC}
		Our aim in the present section is to give sufficient conditions for a strongly minimal ACVF-relic, $\CG$, expanding a group $(G,\oplus, e)$ to interpret a copy of $\CK$.   Specifically, we prove: 
		\begin{proposition}\label{proposition field given infinitely slopes}
			Let $\CK\models\mathrm{ACVF}$ be a saturated model. Let $(G,\oplus)$ be a definable group, locally equivalent to $(K,+)$ or to $(K^*, \cdot)$. Let $\CG$ be a strongly minimal reduct of the induced structure on $G$, expanding the given group structure. Assume that there exists a $\CG$-definable strongly $t$-analytic $C\sub G^2$ and a strongly minimal  $C$-generated family $\CY$ of plane curves with infinitely many 1-slopes at a point $x_0\in G^2$. Then $\CG$ interprets a field, $\CK$-definably isomorphic to $(K,+,\cdot)$. 
		\end{proposition}
		
		We note that the assumption that $\CY$ is strongly minimal (or of Morley Rank 1) is redundant, but since it is all we need, we avoid the more general, but slightly more technical, case. We now proceed to the proof of the proposition, to which the rest of this section is dedicated. 
		
		Our aim is to construct a field configuration. Recall: 
		\begin{definition}[Group configuration]\label{def-gconf}
			\label{def:grconf}
			
			Let $\mathcal M$ be strongly minimal.
			The set $\{a,b,c,x,y,z\}$ of elements in $\CM^{eq}$ is a
			group configuration if there exists an integer $n$ such that
				\begin{itemize}
					\item All elements of the diagram are pairwise independent and
					$\dim(a,b,c,x,y,z)=2n+1$;
					\item $\dim a = \dim b = \dim c=n$,
					$\dim x = \dim y = \dim z = 1$;
					\item The following dependencies hold and no other: $\dim(a,b,c)=2n$, 
					$\dim(a,x,y)=\dim(b,z,y)=\dim(c,x,z)=n+1$.
				\end{itemize}
			A group configuration $\{a,b,c,x,y,z\}$ is \emph{minimal} if given $a'\in \acl(a)$, $b'\in \acl(b)$ and $c'\in \acl(c)$ such that $\{a',b',c',x,y,z\}$ is a group configuration then $a\in \acl(a')$, $b\in \acl(b')$ and $c\in \acl(c')$.
			\end{definition}
			Hrushovski showed that if a strongly minimal set admits a minimal $n$-dimensional group configuration, then there exists an $n$-dimensional group $G$ interpretable in $\CM$ acting transitively on a definable strongly minimal homogeneous space $X$ together with $g_a, g_b\in G$, independent generics such that $a$ is inter-algebraic with $g_a$, $b$ is inter-algebraic with $g_b$ and $c$ is inter-algebraic with $g_ag_b$. Hrushovski showed, further, that $n=1,2,3$ are the only possibilities, and that if $n=2$ then $G$ is $\mathrm{AGL}_1(F)$ for some interpretable algebraically closed field $F$, acting on $(F,+)$.  A minimal 2-dimensional group configuration is, thus, usually referred to as a \emph{field configuration}. 
			We refer to \cite[\S 4.1]{HS} for references and a more detailed discussion. \\

            We leave it for the reader to verify that if $a,b$ are generic independent in  $\mathrm{AGL}_1(F)$ and $x\in F$ is generic independent over $a,b$ then $\{a,b,ab, x, a\cdot x,  b^{-1}\cdot x\}$ is a field configuration. 
						
			We now turn to the proof of the proposition: 
			
			\begin{proof}
				It follows immediately from Hrushovski's Field Configuration Theorem that if we can find a group configuration in $\CG$ that is inter-algebraic in $\CK$ with a standard group configuration\footnote{By inter-algebraic group configurations, we mean group configurations where the respective elements are pairwise inter-algebraic}, as described in the previous paragraph, of the action of $\mathrm{AGL_1}(\CK)$ on $K$ (also denoted $\mathbb G_a \rtimes \mathbb G_m$), then an infinite algebraically closed field is interpretable in $\CG$. That this field is $\CK$-definably isomorphic to $(K,+,\cdot)$ is immediate from, e.g., \cite[Proposition 6.23]{HruRid} and the fact that $G$ (and thus also the interpreted field) are in finite-to-finite correspondence with a definable subset of $K$. So we aim to construct such a group configuration\footnote{Our strategy (similar to several other field construction proofs) is to use the action of composition and $G$-multiplication  on a local group of slopes of  curves in the family to find a group configuration equivalent to one in $\mathbb G_a \rtimes \mathbb G_m$. For power series agreeing up to order $n>1$, composition acts by addition on the $(n+1)$-st coefficients, so we cannot hope that this strategy could work for $n>1$. As, by the aforementioned theorem of Hrushovski's, any field configuration in a strongly minimal structure is necessarily equivalent to a standard group configuration of  $\mathbb G_a(F) \rtimes \mathbb G_m(F)$ for some interpretable field, $F$,  to go through, our strategy seems to require the assumption of infinitely many 1-slopes}. 
				
				Before starting, let us point out that if $C$ is $\CG$-definable, there is a finite subset $S_C$ ($\CK$-definable over the same parameters as $C$) such that $C\setminus S_C$ is strongly $t$-analytic. Obviously, $C\setminus S_C$ is also $\CG$-definable, although we may need additional parameters.

				Next, if  $\CY:=\{Y_s\}_{s\in S}$ is as in the statement of the proposition, and $Y:=Y_{s_0}$ is a generic curve in the family, then, since $\CY$ has infinitely many first order slopes at some point $a$, the function $f_s$ provided by the definition of infinitely many first order slopes is invertible at $a$ with smooth inverse. For future reference, we also note that, removing finitely many points from $Y$, we may further assume that $Y^{-1}$ is $t$-smooth at every point. Note that, by assumption, $Y$ and $Y^{-1}$ are $t$-smooth at $a$, so $a$ is not one of the points we removed.  Therefore,  $\CY_1:=\CY\circ Y^{-1}=\{Y_2\circ Y^{-1}: Y_2\in \CY\}$ is nearly $C$-generated and has infinitely many first order slopes at the point $(x_a,x_a)$ where $a=(x_a,y_a)$. Moreover, the set of first order slopes of curves in $\CY_1$ at $(x_a,x_a)$ contains a ball around $1$. Translating in $G^2$, we may further assume that $x_a=e$ (the unit of $G$). Thus, we have a definable almost faithful, strongly minimal nearly $C$-generated family of curves $\CY_1$ smooth at $(e, e)$ and with infinitely many first order slopes there. We also note that, if $\{f_s: s\in S_1\}$ is a definable family of functions witnessing the fact that $\CY_1$ has infinitely many first order slopes at $(e, e)$, then the set $W_1:=\{f'_s: s\in S_1\}$  is infinite, and (by the chain rule) has $1$ in its relative interior. 
				
				Note that, by the chain rule, if we denote $M(x,y):=x\ominus y$, then $D_xM(c,c)=-D_y(c,c)$ for all $c\in G$. Indeed, $M(c,c)'=0$ so that $D_xM(c,c)+D_yM(c,c)=0$, with the desired conclusion. Thus, if $f, g$ are  definable differentiable functions (with domain and image contained in $G$), then differentiating $M(f(x), g(x))$ we get that $(M(f,g))'(c)=D_xM(f(c),g(c))f'(c)+D_yM(f(c), g(c))g'(c)$, and if $c=f(c)=g(c)=e$ then, choosing $f,g$ with $v(f'(e)-g'(e))$   large enough, $(M(f,g))'(e)$ is arbitrarily close to $0$ which implies that $0$ is in the closure of the set $\{(M(f,g))'(e): f,g\in \CY_1$\}. By considering $\CY_0:=\CY_1\ominus Y_{s'}$ for some $Y_{s'}\in \CY_1$ generic over all the data, we get a definable almost faithful family of curves, each of which contains the  graph of a definable smooth function around  $(e,e)$ and with infinitely many first order slopes.  Moreover, if $\{g_s\}_{s\in S_0}$ is the family of functions witnessing it, $W_0:=\{g_s'(e): s\in S_0\}$, and $0\in W_0$ is in the relative interior pf $W_0$. 
				
                We can find $x_1,x_2\in W_1$ and $y_1,y_2\in W_0$ independent generics and close enough to $0$ and to $1$, respectively, so that in $\mathrm{AGL}_1(\CK)$ we have $(x_1,y_1)(x_2,y_2)=(x_3,y_3)\in W_1\times W_0$. By definition of $\CY_i$ we can find $s_{x_1}, s_{x_2}\in  S_{1}$ and $t_{y_1},t_{y_2}\in S_0$ such that $f_{s_{x_1}}'(e)=x_1$, $g_{t_{y_1}}'(e)=y_1$,  $f_{s_{x_2}}'(e)=x_2$, and $g_{t_{y_2}}'(e)=y_2$. The choice of $(x_i,y_i)$ assures that there are also $s_{x_3}\in  S_1$, $t_{y_3}\in S_0$  such that $f'_{s_{x_3}}(e)=x_3$ and $g'_{t_{y_3}}(e)=y_3$. 
				
				Since the differential of $A(x,y):=x\oplus y$ at $(0,0)$ is $x+y$, we immediately obtain, using the chain rule, that $x_3=(f_{s_{x_1}}\circ f_{t_{y_1}})'(e)$ and $y_3=((f_{s_{x_2}}\circ g_{t_{y_1}})\oplus f_{t_{y_2}})'(e)$. 
				
				Similarly, for $w_1\in W_0$ generic and independent over all the data, if $w_1$ is close enough to $0$ there are $w_2,w_3\in W_0$ such that $(x_1,y_1)\cdot w_1=w_2$ and $(x_3,y_3)\cdot w_3=w_1$ (here $\cdot$ denotes the action of $\mathrm{AGL}_1(\CK)$ on the additive group of $\CK$). In other words, 
				\[\tag{*}\label{e}
				\{(x_1,y_1), (x_2,y_2), (x_3, y_3), w_1,w_2,w_3\}
				\]
				is a standard group configuration for the action of $\mathrm{AGL_1}(\CK)$ of $\CK$ with $(x_i,y_i)\in W_1\times W_0$ and $w_i\in W_0$.  As before, we can find $t_{w_1}, t_{w_2}$ and $t_{w_3}$, all in $S_0$ such that $g_{t_{w_i}}'(e)=w_i$ for $i=1,2,3$. Thus, 
				\[
				\{(s_{x_1},t_{y_1}), (s_{x_2},t_{y_2}), (s_{x_3}, t_{y_3}), t_{w_1},t_{w_2},t_{w_3}\}
				\]
				is a group configuration equivalent to (\ref{e}) (because $\CY_1$ and $\CY_2$ both have infinitely many 1-slopes each value of derivative at $e$ is obtained finitely many times in each of the families $\{f_s\}_{s\in S_1}$ and $\{g_s\}_{s\in S_0}$). It will suffice to show that this is a group configuration in $\CG$. 
				
				To keep the notation readable,  if $s\in S_1$ we write $Y_s$ for the corresponding curve in $\CY_1$ and for $t\in S_0$ we let $Y_t$ denote the corresponding curve in $\CY_0$.  We have already seen that $f_{s_{x_3}}'(e)=(f_{s_{x_1}}\circ g_{t_{y_1}})'(e)$, that is, $Y_{s_{x_3}}$ and $Y_{s_{x_1}}\circ Y_{t_{y_1}}$ have the same first order slope at $(e,e)$. By the paragraph concluding Section \ref{ss plane curves} the intersection $Y_{s_{x_3}}\cap Y_{s_{x_1}}\circ Y_{t_{y_1}}$ is finite. So their multiplicity of intersection at $(e,e)$ is at least $2$. Since $\CY_0$ is a $\CG$-definable family of a-curves smooth at $(e,e)$ with $Y_{s_{x_3}}$ a generic member of $\CY_0$  and  since $Y_{s_{x_1}}\circ Y_{t_{y_1}}$ is a generic element of the nearly $C$-generated family $\CY_1\circ Y_{t_{y_1}}$, we may apply Corollary \ref{C: definability of tangency} to conclude that there are infinitely many $s\in S_1$ such that $|Y_s\cap Y_{s_{x_1}}\circ Y_{t_{y_1}}|>|Y_{s_{x_3}}\cap Y_{s_{x_1}}\circ Y_{t_{y_1}}|$. Strong minimality of $S_0$ implies that the set of $s$ such that $|Y_s\cap Y_{s_{x_1}}\circ Y_{t_{y_1}}|=|Y_{s_{x_3}}\cap Y_{s_{x_1}}\circ Y_{t_{y_1}}|$ is finite, so $s_{x_3}\in \acl_\CG(s_{x_1}, t_{y_1})$. 
				
				An exactly similar argument will show that $s_{y_3}\in \acl(s_{x_2}, t_{y_1}, t_{y_2})$ proving the required  dependence of the edge $(s_{x_1},s_{y_1}), (s_{x_2},s_{y_2}), (s_{x_3}, s_{y_3})$ of the configuration. 
				
				Again, a similar argument would show the other necessary dependencies in our configuration. Noting, for example, that the curve $Y_{s_{w_2}}$ is generic in $\CY_0$ and tangent to the $C$-generated curve $Y_{s_{x_1}}\circ Y_{s_{w_1}}\oplus Y_{s_{w_2}}$. 
				
				This concludes the proof of the proposition 
			\end{proof}

	\section{Power Series}\label{PS}

We aim to prove Theorem \ref{main} separately for the  additive and the multiplicative cases. In the additive case, we  construct a definable family of curves with infinitely many 1-slopes at $(0,0)$, allowing us to apply Proposition \ref{proposition field given infinitely slopes} to conclude. The present section contains most of the technical power series machinery  we  need  to construct such a family. Throughout this section we are working in a complete algebraically closed valued field $\Kk$ and  put $p=\mathrm{char}(\Kk)$.




\begin{lemma}\label{existence of infinely many N slopes Lemma}
 Let $x_0\in K$ and  $\mathcal F=\{f_s\}_{s\in Q}$ an infinite definable family of functions, each analytic in a ball $B\ni x_0$. Assume that $f_s(x_0)=f_t(x_0)$ for all $s, t \in Q$, and that the family of graphs $\{(x,f_s(x)): s\in S, x\in B\}$ is almost faithful. Then, denoting
 \[
 f_s(x)=x_1+\sum_{n\geq 0}e_n(s)(x-x_0)^n,
 \] 
the power series expansion of $f_s$ at $x_0$, there is some $N$ such that the set $\{e_N(s)\mid s\in Q\}$ is infinite.
    \end{lemma}

\begin{proof}
    If the statement holds for an infinite subset of $Q$, it holds for $Q$ so we may assume that $Q$ is  one dimensional, and by shrinking and projecting  into $K$, we may assume $Q\sub K$ (see, e.g., \cite[Lemma 6.4]{SimWal}). Shrinking further, and rescaling, we may assume $Q=\mathcal O$.  Assume the claim fails. For all $n$ define an equivalence relation $E_n$ on $Q$ by setting $E_n(s,t)$ is $f_s$ and $f_t$ have the same polynomial approximation of order $n$. Our assumption implies that $E_n$ has finitely many classes for all $n$. 

    Since the residue field, $k$ is strongly minimal with the $\Kk$-induced structure there is, for every $n$, at least one $E_n$-class whose image in $k$ is co-finite. Let $S_1\sub Q$ be such a class. Assume we have defined $S_n$ such that $S_n\sub S_{n-1}$,  $S_n$ is contained in a single $E_n$ class, and the image of $S_n$ in $k$ is co-finite. The $E_{n+1}$-classes cover $S_n$, so the intersection of one of them with $S_n$ has co-finite image in $k$. We let $S_{n+1}$ be the intersection of one such class with $S_n$. 

    Consider the type-definable set $S:=\bigcap S_n$. Its image in $k$ is co-countable, and since $\Kk$ is uncountable it follows that $S$ is infinite. But for any $s,t\in S$ the functions $f_s$ and $f_t$ have the same power series expansion around $x_0$, so by the Uniqueness Theorem the two functions are equal. Since $S$ is infinite, this contradicts our assumption that the family of graphs of the functions $\{f_s\}_{s\in S}$ is almost faithful.     
 \end{proof}

We will mainly use families of functions obtained by translating a given function either by the additive or by the multiplicative group. 

We start with a well known fact, the proof is easy and can be found, e.g., in   \cite[Proposition 1.4.3]{Cher}.

\begin{fact}\label{coefficientsAreAnalyticLemmaFact}
   Let $U\ni 0$ be open and let $f:U\to \Kk$ be an analytic function. For $a\in U$ and $m\in \mathbb N$ let $e_m(a)$ be the coefficient of  $(x-a)^m$ in the power expansion of $f$ at $a$, so
    $$f(x)=\sum_{m\geq 0} e_m(a) (x-a)^m.$$
    Then the function $a\mapsto e_m(a)$ is analytic on $U$.
\end{fact}

\begin{defi}\label{defN}
Let $$f(x)=\sum_{n\geq 1} b_n x^n$$ be a power series converging in a neighborhood $D(f)$ of $0$. For $a\in D(f)$ define $f_a(x)=f(x+a) -f(a)$. If the power series expansion of $f_a$ around $0$ is $f_a(x)=\sum_{n\geq 1}b_{n,a} x^n$, define
\begin{equation*}
\begin{split}
s_n(f) &:=\{b_{n,a}\mid a\in D(f)\}  \\ 
N_+(f) &:=\min\{n\mid s_n(f)\text{ is infinite}\}.
\end{split}
\end{equation*}
\end{defi}

\begin{fact}\label{N}
Let $f(x)$ be a function analytic at $0$ such that $f(0)=0$ and $f'(0)\neq 0$. Then $f$ is injective in a neighborhood of $0$; let $g(x)$ be its analytic inverse around $0$.
Let $$f(x)=\sum_i b_i x^i$$ and 
$$g(x)=\sum_j c_j x^j$$  be the power expansions of $f$ and $g$, respectively.

Then for each $n$, $c_n$ depends only on $b_1,\ldots,b_n$. Moreover, if $N_+(f)$ is finite, $N_+(g)\geq N_+(f)$.
\end{fact}

\begin{proof}

The formula $$x=f(g(x))=\sum_i b_i\left(\sum_j c_j x^j\right)^i$$ implies that $c_1=1/b_1$ and, by induction, that $c_n$ depends only on $b_1,b_2,\ldots,b_n$. 
\end{proof}

\begin{corollary}
If $N_+(f)$ and $N_+(g)$ are both finite, then $N_+(f)=N_+(g)$.
\end{corollary}
\begin{proof}
By Fact \ref{N} one has that $N_+(f)\leq N_+(g)$ apply the same lemma for $g$ and get that $N_+(g)\leq N_+(h)$ where $h$ is the inverse of $g$. Since $h=f$ in some open neighborhood of $0$,  $N_+(f)=N_+(g)$.
\end{proof}

\begin{lemma}\label{lemmaN}  Let 
\[
f=\sum_{n\geq 1}b_n x^n
\] 
be a non-additive analytic function converging in a neighborhood of $0$. 

Then 
\[
E:=\{e\in \mathbb N:\exists n>1 \text{ such that }p\nmid n\wedge  b_{np^e}\neq 0\}
\]
is not empty, and if $l=\min E$, then the coefficient $b_{mp^k,a}$ does not depend on $a$ for all $k<l$ and all $m\geq 1$ such that $p\nmid m$. Moreover $N_+(f)=p^l$. 
\end{lemma}

Recall that $b_{mp^k,a}$ is the coefficient of $x^{mp^k}$ in the expansion of $f_a(x)=f(x+a) - f(a)$ around $0$. 
\begin{proof}
If $E=\0$ 
then the power series expansion of $f$ consists only of powers of $p$,  implying  that $f(x+a)=f(x)+f(a)$, contradicting our assumption that $f$ is not additive. Thus, $E\neq \0$.

Denoting $l=\min E$ write 
\[
f(x)=b_1 x +b_p x^p+ b_{p^2} x^{p^2} + \ldots + b_{p^l} x^{p^l} + \sum_{i>p^l} b_i x^i.
\]  

\begin{claim}
 If $k<l$ and $m\geq 1$ with $p\nmid m$, then the coefficient $b_{mp^k,a}$ is constant as $a$ varies.   
\end{claim}

\begin{claimproof}
 Fix $t\in \mathbb N$. Since $f_a(x)=f(x+a) - f(a)$ and $f(x+a)=\sum_{n\geq 1}b_n (x+a)^n$, we need to analyze 
 each occurrence of the term  $x^{mp^k}$ in $(x+a)^t$. For this, it is of course enough to assume $t\geq mp^k$. 
 Let $t=u p^s$ with $p\nmid u$, so that $(x+a)^t=(x^{p^s} + a^{p^s})^u$.

If $u>1$ then by definition of $l$ we have $b_s=0$ for $s<l$. If $s\ge l$ then, by hypothesis, $l>k$. Since all powers of $x$ in the expansion of $(x+a)^t$ are powers of $x^{p^s}$, no monomial of $x^{mp^k}$ appears with a non-zero coefficient. 

If $u=1$ then $t=p^s$ and $(x+a)^t=x^{p^s} +a^{p^s}$. So the monomial $x^{mp^k}$ appears in the expansion of $(x+a)^t$ only when $k=s$ and $m=1$. 

It follows that the coefficient of $x^{mp^k}$ in $f_a(x)$ is $b_{mp^k}$ which is independent of $a$. 
\end{claimproof}

So in order to prove that $N_+(f)=p^l$, we just need to show that the coefficient $b_{p^l,a}$ is not locally constant as $a$ varies. As 
\[
f(x)=b_1 x + b_p x^p+ b_{p^2} x^{p^2} + \ldots + b_{p^l} x^{p^l} + \sum_{i>p^l} b_i x^i
\] 
and $b_{np^k}=0$ for all $n>0$ and $k<l$ such that $p\nmid n$, the function $\displaystyle\sum_{i>p^l} b_i x^i$ is, in fact, a function in the variable $x^{p^l}$. So 
\[
f(x)=b_1 x +b_p x^p+ \ldots + b_{p^l} x^{p^l} +\sum_{i>1} b_{i p^l}x^{ip^l}
\]
which implies that the coefficient of $x^{p^l}$ in $f_a$ is $$b_{p^l,a}=b_{p^l} + \sum_{i>1} i b_{i p^l} (a^{p^l})^{i-1} $$ and since $b_{np^l}\neq 0$ for some $n>1$ with $p\nmid n$, this is a series in $a$ with some non-zero coefficient. By the Identity Theorem (Fact \ref{identityTheoremf}) any such series cannot be constant as $a$ varies in any open neighborhood of $0$, so it takes infinitely many values in any open neighborhood of $0$. 
\end{proof}

We will also need:

\begin{lemma}\label{compositionMultLemma}
 Let $U\ni 1$ be open and $h:U\to K$ and $g:U\to K$ be analytic functions. 
 
Let
$$h(x)=1+\sum_{n\geq 1} a_n (x-1)^n,$$  
$$g(x)=1+\sum_{n\geq 1} b_n (x-1)^n$$ 
and 
$$(h\circ g)(x)=1+\sum_{n\geq 1} c_n (x-1)^n.$$ 
Then for each $N$, $c_N$ depends only on $a_j$ and $b_j$ for $j\leq N$.

In addition, if $a_1=b_1=1$ and $N\geq 2$ is such that $a_n=b_n=0$ for $1<n<N$,  then $c_N=a_N+b_N$.

\end{lemma}

\begin{proof}
Notice that 
$$h\circ g(x)=1+\sum_{n\geq 1} a_n \left(\sum_{m\geq 1} b_m(x-1)^m\right)^n$$
and the coefficient of $(x-1)^N$ on the right side of this expression is 
\begin{equation}\label{composeCoefEquation}
    c_{N}= a_1 b_{N} + a_2 f_2 + \ldots +a_{N-1}f_{N-1} + a_N b_1^N
\end{equation}
where for $n=2,\ldots, N-1$, 
$$f_n=\sum_{i_1+\ldots+i_n=N} b_{i_1}\cdots b_{i_n},$$ 
which only depends on $b_1,\ldots,b_N$. So $c_N$ only depends on $a_1,\ldots,a_N$ and $b_1,\ldots,b_N$, as required.

If $a_n=b_n=0$ for all $1<n<N$, then $f_2=f_3=\ldots=f_{N-1}=0$ and Equation \ref{composeCoefEquation} implies
$$c_N = a_1 b_N+ a_N b_1^N$$ and 
$$c_N=a_N+b_N$$ if $a_1=b_1=1$.
\end{proof}

\section{The additive case}\label{TAC}

In this section, we prove:
 \begin{proposition}\label{proposition additive case}
  
  Let $\CK\models \mathrm{ACVF}$ and let $\mathcal G:=(G,\oplus, C \dots)$ be a non locally modular strongly minimal $\CK$-relic expanding a group $G$ such that $C\subseteq G\times G$ is a strongly minimal non $\mathcal G$-affine plane curve. 
  
  Assume that $G$ contains a finite index $\CK$-definable subgroup which is $\CK$-definably isomorphic to an infinite subgroup of $(K,+)$.
  
  Then $\mathcal G$ interprets a field $\mathcal K$-definably isomorphic to $(K,+,\cdot)$. 
\end{proposition}

We introduce  a convenient definition:
\begin{definition}
    Let $C\subseteq K\times K$ be a plane curve.  A point $p\in C$  is \emph{$y$-regular in $C$} if there is an irreducible polynomial $F(x,y)$ such that $\frac{\partial F}{\partial y}(p)\neq 0$ and $p$ is an interior point of the intersection of $C$ with the zero set of $F$.
\end{definition}

 Notice that the statement ``$p$ is a $y$-regular point of $C$'' is uniformly $\CK$-definable in definable families. 

Fix $G$ and $\mathcal G$ as in the hypothesis. The assumption that $G$ has a finite index subgroup definably isomorphic to a subgroup of $(K,+)$ implies, in particular, that $G$ is locally equivalent to $(K,+)$, so as explained in Section \ref{ss: the setting} we may apply the results of the previous sections.

Let $\{t_a(C)\mid a\in C\}$ be the family of translates of $C$ by elements of $C$. Since $C$ is strongly minimal and non-affine, this is a strongly minimal almost faithful family of plane curves. 

We start with an easy remark:

\begin{remark}\label{R: 0 is y smooth and C is t-analytic }
By Proposition \ref{decompositionProp} $C$ is locally Zariski closed, so after removing finitely many isolated points we may assume that each $z\in C$ is the interior of the intersection of $C$ with the zero set of one of finitely many irreducible polynomials $F_1,\dots, F_l$ (some $l$). 

As  the zero sets of the $F_i$ have only finitely many singular points, we may take some $d\in C$ regular. So either $\frac{\partial F_1}{\partial y}(d)\neq 0$ or $\frac{\partial F_1}{\partial x}(d)\neq 0$. 

In the former case, replace $C$ by $t_d(C)$, in the latter case,  replace $C$ with $t_d(C)^{-1}$. Either way, we may assume that $(0,0)\in C$ is a $y$-regular point.

 Moreover, as explained in the discussion following Definition \ref{D: strongly t-an}, after removing finitely many points (but not $(0,0)$) from $C$ we may also assume that $C$ and $C^{-1}$ are both strongly $t$-analytic.
\end{remark}

To prove Proposition \ref{proposition additive case}, we will show that for some $c\in C$, denoting 

\begin{equation}\label{definition of Y}
Y_{a,c}=(C-t_a(C))\circ (C-t_c(C))^{-1},
\end{equation}
the family $\mathcal Y=\{Y_{a,c}\}_{a\in C}$ has infinitely many 1-slopes at $(0,0)$. We then show that after removing finitely many points from $C-t_c(C)$ (not including $(0,0)$), $\CY$ is a nearly $C$-generated family, allowing us to conclude using  Proposition \ref{proposition field given infinitely slopes}.  

\medskip

Most of the work requires results from Section \ref{PS}, proved for complete models. So thus prove, first: 

\begin{lemma}\label{infinitely many slopes complete model}
    Let $\Kk$ be complete, $U\ni 0$ open and  $h:U\to G$ a non-additive analytic function with $h(0)=0$.  For $a\in U$ let $h_a(x):=h(a+x)-h(a)$ and let $h_a(x)=\sum_{n\geq 1}e_n(a) x^n$ its power series expansion around $0$. 
    
    Then, there is some $c\in U$ such that $h-h_c$ is invertible, and the family $(h-h_a)\circ (h-h_c)^{-1}$ has infinitely many 1-slopes at $0$ as $a$ varies in $U$.
\end{lemma}
\begin{proof}
It is enough to show:

    \begin{claim}
        After maybe shrinking $U$ there is an element $c\in U$, and a map $H:U\times U\to G$ such that:

\begin{enumerate}
    \item $H(0,s)=0$ for all $s\in U$.

        \item For all $s\in U$ the map $x\to H(x,s)$ is analytic in some open set around $0$ and moreover, for a fixed such $s$
        \[
        H(x,s)=(h-h_s)\circ (h-h_c)^{-1} (x)
        \]
        on their common domains.
     \item The set $$s_1(H):=\left\{\frac{\partial H}{\partial x}(0,s)\mid s\in U\right\}$$ is infinite.
\end{enumerate}
    \end{claim}

\begin{claimproof}
As $h$ is not additive Lemma \ref{lemmaN} implies $N_+(h)=p^m$ for some $m$. By definition, for $n<p^m$ the set $s_n(h)=\{d_n(a)\mid a\in U\}$ (defined on Definition \ref{defN}) is finite. After shrinking $U$, we may assume that 
$d_n(a)=d_n$ for all $a$ in $U$ and $n<p^m$. This implies that all $h-h_a$ has a zero of order at least $p^m$ for all $a\in U$, i.e.,    
$$(h-h_a)(x)=\sum_{i\geq 1}e_{i}(a)x^{ip^m}$$
where $e_{i}(a)=d_{ip^m}-d_{ip^m}(a)$.
By definition of $N_+(f)$, $\{e_{1}(a)\mid a\in U\}$ is infinite, so we can choose $c\in U$ such that $e_1(c)\neq 0$.

For  $a\in U$ let $$G_a(x)=\sum_{i\geq 1}e_{i}(a) x^i$$ so $$(h-h_a)(x)=G_a(x^{p^m}).$$ As $e_{1}(c)\neq 0$ we have $G'_c(0)\neq 0$ and Fact \ref{inversionf} implies that there is an analytic function $G^{-1}_c$ defined in some neighborhood of $0$ such that $G_c(G_c^{-1}(z))=z$ for all $z$ in that neighborhood.

\medskip

        Let $y=G_a(G_c^{-1}(x))$ and $z=\Fr^{-m}(G_c^{-1}(x))$. Then, $$(h-h_c)(z)=G_c(z^{p^m})=G_c(G_c^{-1}(x))=x$$ and  $$(h-h_a)(z)=G_a(z^{p^m})=G_a(G_c^{-1}(x))=y.$$ So $y=(h-h_a)\circ (h-h_c)^{-1} (x)$, so that the graph of $G_a \circ G_c^{-1}$ is contained in the graph of $(h-h_a)\circ (h-h_c)^{-1}$.

\medskip

Now note that  
$$G_c^{-1}(x)=\frac{1}{e_{1}(c)}x + L(x)$$ 
for some analytic function $L$ having a zero of degree at least $2$ at $x=0$, which implies $$G_a\circ G_c^{-1} (x) = \frac{e_{1}(a)}{e_{1}(c)} x + \sum_{i>1} f_{i}(a) x^{i}$$ for some coefficients $f_{i}(a)$. By Fact \ref{coefficientsAreAnalyticLemmaFact}, each $f_{i}(a)$ is a power series in the variable $a$, and the same is true for $e_1(a)$. Therefore, there is a power series $H(x,a)$ such that $H(x,a)=(G_a\circ G_c^{-1})(x)$ for all $a$ in a neighborhood of $0$, so 
$$\frac{\partial H}{\partial x}(0,a)=(G_a\circ G_c^{-1})'(0)=\frac{e_{1}(a)}{e_{1}(c)},$$ 
which takes infinitely many values as $a$ varies. \end{claimproof}

    This completes the proof of Lemma \ref{infinitely many slopes complete model}.
     \end{proof}

In order to use Lemma \ref{infinitely many slopes complete model} we provide a first order statement which can be proved in any model:

\begin{lemma}\label{lemma infinite slopes}
     Let $C\subseteq G\times G$ be a curve such that the family $\{t_a(C)\mid a\in C\}$ is almost faithful and with $(0,0)$ a $y$-regular point of $C$.

Then, there is some $c\in C$ such that, after removing from $C-t_c(C)$  a $c$-definable finite set not containing $(0,0)$, the family $\{Y_{a,c}\}_{a\in C}$ is nearly $C$-generated and has infinitely many $1$-slopes at $(0,0)$.

\end{lemma}

\begin{proof}
    Notice that this is a first order statement, so we may prove it in a complete model. 

Let $F(x,y)$ be as provided by the assumption that  $(0,0)$ is $y$-regular in $C$. By Fact \ref{implicitFunctionTheoremf} there is an analytic function $h(x)$ defined in a neighborhood $U\ni 0$ such that $h(0)=0$ and $F(x,h(x))=0$ for all $x\in U$. Shrinking $U$ we may assume that $(x,h(x))\in C\cap (U\times U)$ for all $x\in U$. Moreover, if 
\[h_a(x)=h(x+a)-h(a)=\sum_{n\geq 1}e_n(a)x^n\] we have that $(x,h_a(x))\in t_a(C)$ for all $x\in U$. As the family $\{t_a(C)\mid a\in C\}$ is almost faithful, so is the family of graphs of $(h_a)_{a\in U}$. Therefore,  by Lemma \ref{existence of infinely many N slopes Lemma} there is some $N$ such that $\{d_N(a)\mid a\in U\}$ is infinite. Thus, we may apply Lemma \ref{infinitely many slopes complete model} and find some $c\in U$ such that $\{((h-h_a)\circ (h-h_c)^{-1})'(0)\mid a\in U\}$ is infinite. As $Y_{a,c}$ contains the graph of $(h-h_a)\circ (h-h_c)^{-1}$  it implies that the family $\{Y_{a,c}\}_{a\in C}$ has infinitely many $1$-slopes at $(0,0)$. Moreover, as $C-t_c(C)$ contains the graph of $h-h_c$ near $(0,0)$, we have that after removing from $C-t_c(C)$ the finite set of non $t$-smooth points (all different from $(0,0)$), the family $\CY$ is nearly $C$-generated.
\end{proof}

As a corollary we have Proposition \ref{proposition additive case}:
\begin{proof}(Proof of Proposition \ref{proposition additive case})

Let $\{t_a(C)\mid a\in C\}$ the almost faithful family of $G$-translates of $C$. As was pointed out in the discussion after \ref{D: strongly t-an}, after removing finitely many points of $C$ we have that $C$ and $C^{-1}$ are both strongly $t$-analytic curves. Moreover, by Remark \ref{R: 0 is y smooth and C is t-analytic } after taking some translate by a regular point $d\in C$  and maybe replacing $C$ by $C^{-1}$ we may assume that $(0,0)$ is a $y$-regular point of $C$. Apply Lemma \ref{lemma infinite slopes} and find some $c\in C$ such that -after removing finitely many points from $C-t_c(C)$ but not $(0,0)$-  the family  $\mathcal Y=\{Y_{a,c}\}_{a\in C}$ is nearly $C$ generated and has infinitely many slopes at $(0,0)$ so we may apply Proposition \ref{proposition field given infinitely slopes} and conclude. \end{proof}

\section{The multiplicative case}\label{TMC}

 In this section, we prove:
 
 \begin{proposition}\label{proposition multiplicative case}
  
  Let $\CK\models \mathrm{ACVF}$ and let $\mathcal G:=(G,\oplus, C \dots)$ be a non locally modular strongly minimal $\CK$-relic expanding the group $(G, \oplus)$, and let  $C\subseteq G\times G$ be a $\mathcal G$-definably  strongly minimal non $\mathcal G$-affine set.
  
  Assume furthermore that $(G,\oplus)$ is either a subgroup of the multiplicative group or $G_{m,\lambda}$ for some $\lambda\in K$. 
  
  Then $\mathcal G$ interprets a field which is $\CK$ definably isomorphic to $(K,+,\cdot)$. 
\end{proposition}

Fix $G$ and $\mathcal G$ as in the statement. 

As in the additive case, since $C$ is locally Zariski closed, after possibly taking a translate and replacing $C$ by $C^{-1}$ we may assume that $(1,1)$ is a $y$-smooth point of $C$. Moreover, after removing finitely many points, none of them $(1,1)$, we may assume that $C$ and $C^{-1}$ are both strongly $t$-analytic.

The proof splits into  two cases. If the family $\{t_a(C):a\in C\}$ of translates of $C$ has infinitely many 1-slopes at $(1,1)$ then we can conclude using Proposition \ref{proposition field given infinitely slopes}. 

Most of this section is devoted to the situation where we have only finitely many 1-slopes at $(1,1)$, proving that in this case we can then reduce to the additive case: 

 \begin{proposition}\label{exists additive group configuration}
 Assume that $(G,\oplus)$ is locally equivalent to the multiplicative group. Let $C\subseteq G\times G$ be a $\mathcal G$-definable strongly minimal with $(1,1)\in C$ $y$-smooth. Assume moreover that $\{t_a(C):a\in C\}$ is  almost faithful  with finitely many $1$-slopes at $(1,1)$. Then there is a $\CG$-definable group locally equivalent to $\mathbb G_a$. 
 \end{proposition}

 As in the proof of Lemma \ref{infinitely many slopes complete model}, we will prove the following technical result: 

\begin{lemma}\label{Lemma infinite N slopes complete}
Let $\Kk$ be complete, $U\ni 1$ open and $h:U\to G$ a non-constant 
analytic with $h(1)=1$. For $s\in U$ let $h_s(x):=h(sx)/h(s)$ and let $h_s(x)=1+\sum_{n\geq 1}e_n(s) (x-1)^n$ its power series expansion around $1$. Let  
\[
    E_\times (h):=\{n\in \mathbb N: |\{e_n(s):s\in U\}|=\infty\},
\] and assume that $E_\times (h)\neq \emptyset$ and $N_\times(h):=\min E_\times (h)>1$.

Then there is some $c\in U$ be such that $h_c$ is invertible at $1$ and, if for $s\in U$ we let $f_a:=h_s\circ h_c^{-1}$, the function $f_s$ is analytic at $1$ and there is some $N\in \mathbb N$, such that:
    \begin{itemize}
        \item $f_s^{(1)}(1)=1$ for all $s\in U$,
        \item $f_s^{(n)}(1)=0$ for all $s\in U$ for all $1<n<N$, and
        \item $\{f^{(N)}_s(1):s\in U\}$ is infinite,
    \end{itemize} where, abusing notation, we let  $f^{(n)}(1)$ denote the $n$-th coefficient of the power series expansion of $f$ around $1$.
    \end{lemma}

\begin{proof}
    It will suffice to show:

     \begin{claim}
        There are $c$ arbitrarily close to $1$, an open $V\ni 1$, a function \[H(x,s)=1+\sum_{n\geq 1}d_n(H,s)(x-1)^n\] and $N\in \mathbb N$ such that $h_c$ is invertible at $1$ and:
        \begin{enumerate}
            \item The map $x\mapsto H(x,s)$ is analytic for each $s\in V$.
             \item  $H(x,s)=h_s\circ h_c^{-1}(x)$  for all $s\in V$ and $x$ in their common domain.
            \item $d_1(H,s)=1$ and $d_n(H,s)=0$ for all $s\in V$ and all $1<n<N$.
            \item $\{d_N(H,s): s\in V\}$ is infinite.

        \end{enumerate}
            
        \end{claim}

        \begin{claimproof} For $s\in U$ define $H_0(x,s):=h_s(x)$. Note that by continuity of multiplication, there is $V\ni 1$ open such that $V\cdot V\subseteq U$ so $h_s(x)=h(sx)/h(s)$ converges in $V$ for all $s\in V$.
 
      Write the expansion of $h_s(x)$ around $1$ as: 
    
    $$h_s(x)=1+\sum_{n\geq 1} d_n(H_0,s) (x-1)^n.$$

By construction, $N_\times(h)$ is  the minimal natural number $n$ such that $d_{n}(H_0,s)$ is infinite as $s$ varies in $V$. 



    Shrinking $V$ we may assume that for $n<N_\times(h)$, $d_n(H_0,s)$ is constant as $s$ varies in $V\setminus\{1\}$.     
    
    Let $r$ be the maximal natural number ($r$ may be $0$) such that there exists some function $g(x)$ analytic in some neighborhood of $1$ with $$h(x)=g(x^{p^r}).$$  
    Shrinking $V$ we assume that $g$ is analytic in $V$.    
    By maximality of $r$, and the fact that $h$ is non-constant, $g$ is not a function in the  $x^p$ which implies by the Identity Theorem (Fact \ref{identityTheoremf}) that  $g'(x)$ is not locally zero. So $g'(x)=0$ 
    for only finitely many points in $V$. Thus, we can find $c_0\in V$ arbitrarily close to $1$ such that both $g'(c_0)\neq 0$ and $d_{N_\times(h)}(H_0,\Fr^{-r}(c_0))\neq 0$; by our choice of $N_\times(h)$, both are open and non-empty conditions, which implies that such a $c_0$ exists. 
Let $c:=\Fr^{-r}(c_0)$.

    As before, let $g_a(x)=g(ax)/g(a)$. $$g'_{c_0}(1)=g'(c_0)\frac{c_0}{g(c_0)}\neq 0,$$ so Fact \ref{implicitFunctionTheoremf} implies that $g_{c_0}$ has a local inverse $f$ converging in an open neighborhood of $1$. Shrinking $V$ we may assume that both, $f$ and $g$, converge in $V$.

    Notice that $$h_c(x)=g((cx)^{p^r})/g(c^{p^r})=g(c_0 x^{p^r})/g(c_0)=g_{c_0}(x^{p^r})$$ and as $f$ is a local inverse for $g_{c_0}$ then $\Fr^{-r}\circ f$ is an inverse for $h_c$ showing that $h_c$ is invertible. 

  We now show that the function
   $$H(x,a):=g_{a^{p^r}} \circ f (x) = g(a^{p^r} f(x))\frac{1}{g(a^{p^r})}$$  satisfies the conclusion of the claim.

   Shrinking $V$ further, we may assume that $x\mapsto H(x,a)$ is analytic in $V$ for all $a\in V$. This gives Clause 1. 

\bigskip

   We proceed to proving Clause 2. Namely, that for $a\in V$ we need to show that  $H(x,a)=h_a\circ h_c^{-1}(x)$ for all $x$ in their common domains.

   Fix $a,x_0\in V$ and let $y_0:=H(x_0,a)$ we prove that $y_0=h_a\circ h_c^{-1}(x_0)$. Let $z:=f(x_0)$ so, by definition of $f$, 
   \[g_{c_0}(z)=g(c_0z)/g(c_0)=x_0\] 
   and \[g_{a^{p^r}}(z)=g(a^{p^r}z)/g(a^{p^r})=g_{a^{p^r}}(z)=g(a^{p^r}f(x))/g(a^{p^r})=H(x_0,a)=y_0.\]
   
   Let $w:=\Fr^{-r}(z)$. Then the above equalities imply
   \[
   h_{c}(w)=h(c w)\frac{1}{h(c)}=g\left((cw)^{p^r}\right)\frac{1}{g(c^{p^r})}=g(c_0z)\frac{1}{g(c_0)}=g_{c_0}(z)=x_0, \text{ and }
   \]
   \[
   h_{a}(w)=h(a w)\frac{1}{h(a)}=g\left((aw)^{p^r} \right)\frac{1}{g(a^{p^r})}=g_{a^{p^r}}(w^{p^r})=g_{a^{p^r}}(z)=y_0,
    \]
   so $ h_c^{-1} (x_0)=w$ and $h_a(w)=y_0$. Thus, $y_0=h_a\circ h_c^{-1}(x_0)$ which completes the proof of Clause 2.

\bigskip

Write $$H(x,s)=1+\sum_{n\geq 1} d_n(H,s) x^n.$$
Now we prove that -- after possibly  shrinking $V$ even more --  $d_1(H,s)=1$ and $d_n(H,s)=0$ for all $s\in V$ and all $1<n<N_\times(h) p^r$.

Write the power series expansion of $g_{a^{p^r}}$ around $1$ as:

By definition of $d_n(H_0,a)$ we have $b_n(a)=d_{np^r}(H_0, a)$ and $b_n(c)=d_{np^r}(H_0,c)$. Moreover, $d_k(H_0,a)=d_k(H_0,c)=0$ for all $k\in \mathbb N$ such that $p^r\nmid k$. As we are assuming that $d_{N_\times(h)}(H_0,c)\neq 0$ it follows that $N_\times(h)$ is a multiple of $p^r$ so $N_\times(h)=N p^r$ for some $N$. We are also assuming that $d_n(H_0,s)$ is constant as $s$ varies on $V\setminus \{1\}$ for $n< N_\times(h)$, so in particular $d_n(H_0,a)=d_n(H_0,c)$ for all $a\in V\setminus \{1\}$ and all $n<N_\times(h)$ 

This implies 
$$h_{a}(x)= 1 + b_1 (x-1)^{p^r} + \ldots  + d_{N_\times(h)}(H_0,a)(x-1)^{N_\times(h)} + \sum_{n > N} d_{np^r}(H_0,a) (x-1)^{np^r}$$
and 
$$h_{c}(x)= 1 + b_1(x-1)^{p^r} + \ldots  +  d_{N_\times(h)}(H_0,a)(x-1)^{N_\times(h)}+ \sum_{n > N} d_{np^r}(H_0,c)(x-1)^{np^r}.$$

Since for $i$ less than $N$ we have
$$b_i:=b_i(a)=d_{ip^r}(H_0,a)=d_{ip^r}(H_0,c)=b_i(c),$$
$$g_{a^{p^r}}(x)=1+b_1(x-1)+\ldots+d_{N_\times(h)}(H_0,a)(x-1)^{N}+\sum_{n>N}d_{np^r}(H_0,a)(x-1)^n$$
and 
$$g_{c_0}(x)=1+b_1(x-1)+\ldots+d_{N_\times(h)}(H_0,c)(x-1)^{N}+\sum_{n>N}d_{np^r}(H_0,c)(x-1)^n.$$

Applying Lemma \ref{N} to $g_{c_0}(x+1)-1$, as $f$ is the inverse of $g_{c_0}$, the coefficient of $(x-1)^n$ in the inverse of $g_{a^{p^r}}$ is the same as the coefficient of $(x-1)^n$ in $f(x)$ for $n < N$. Thus, if the power expansion of $f$ around $1$ is given by:
$$f(x)=1 + c_1 (x-1) +\ldots + c_{N-1}(x-1)^{N-1}+c_{N} (x-1)^{N}+ \sum_{n>N} c_n (x-1)^n,$$ 
the inverse of $g_{a^{p^r}}$ has a power expansion around $1$ of the form
$$g_{a^{p^r}}^{-1}(x)=1 + c_1 (x-1) +\ldots + c_{N-1}(x-1)^{N-1}+ c_{N}(a) (x-1)^{N}+ \sum_{n>N} e_n (x-1)^{n}$$ for some $e_n$.

For $n<N$ the coefficient of $(x-1)^n$ in $f$ is the same as the coefficient of $(x-1)^n$ in $g^{-1}_{a^{p^r}}$, Lemma \ref{compositionMultLemma} implies that for $n<N$ the coefficient of $(x-1)^n$ in 
$$g_{a^{p^r}}\circ f$$ is the same as the coefficient of $(x-1)^n$ in

$$g_{a^{p^r}}\circ g^{-1}_{a^{p^r}}=1+(x-1).$$ 
So
\begin{equation}\label{ln}
    g_{a^{p^r}}\circ f (x) = 1 + (x-1) + d_{N}(H,a)(x-1)^{N}+\sum_{n>N} d_n(H,a)(x-1)^n
\end{equation}
and $d_1(H,a)=1$ and $d_n(H,a)=0$ for all $1<n<N$. This completes Clause 3.

\bigskip

Finally, it follows from Equation \ref{composeCoefEquation} that the coefficient $d_{N}(H,a)$ in the Equation \ref{ln} is given by:

\begin{equation}\label{dNEquation}
    d_{N}(H,a)=b_1 c_{N} + b_2 q_2+\ldots+ b_{N-1} q_{N-1} + d_{N_\times(h)}(H_0,a) c_{1}^{N}
\end{equation}
where for $n=2,\ldots, N-1$

$$q_n=\sum_{i_1+\ldots+i_n=N} b_{i_1}\cdots b_{i_n}.$$

But $q_n$ only depends on $b_1,\ldots,b_{N-1}$ for $n=2,\ldots,N-1$. So the first $N-1$ terms on the right hand side of Equation \ref{dNEquation} only depend on $b_1,\ldots,b_{N-1}$ and $c_1,\ldots, c_N$ all of which are constant as $a$ varies. As $d_{N_\times(h)}(H_0,a)$ takes infinitely many values, so must $d_{N}(H,a)$, and $\{d_N(H,s):s\in V\}$ is infinite. \end{claimproof}

This finishes the proof of Lemma \ref{Lemma infinite N slopes complete}.
\end{proof}

\begin{lemma}\label{lemma first order multiplicative}
      Let  $C\subseteq G\times G$ be a curve such that the family $\{t_a(C):a\in C\}$ is almost faithful. Assume moreover that $(1,1)$ is a $y$-smooth point of $C$ and that the family of translates of $C$ has finitely many $1$-slopes at $(1,1)$. Then there exist an open set $V\ni (1,1)$ and an element $c\in V\cap C$  such that, denoting for $a\in C$   \begin{equation}\label{defXmult}
    Y_{a,c}:=t_a(C)\circ t_c(C)^{-1}, 
\end{equation}
we have:

    \begin{enumerate}
        
        \item The family $(Y_{a,c})_{a\in C}$ is almost faithful.
        \item  
        The family $(Y_{a,c})_{a\in V\cap C}$ has infinitely many $N$-slopes for some $N>0$ and for all $a\in V\cap C$ the $1$-slope of  $Y_{a,c}$ at $(1,1)$ is $1$ and the $n$-slope of $Y_{a,c}$ at $(1,1)$ is $0$ for $1<n<N$.
    \end{enumerate}

    \end{lemma}

    \begin{proof}
         The statement is elementary, so we may prove it in a complete model.  
        Let $F(x,y)$ be as provided by the fact that $(1,1)$ is a $y$-smooth point of $C$. Then as $F(1,1)=0$ and $\frac{\partial F}{\partial y}(1,1)\neq 0$ we may apply Fact \ref{implicitFunctionTheoremf} and find $U\ni 1$ open and $h:U\to K$ analytic such that $(x,h(x))\in C$ for all $x\in U$. As the family of translates of $C$ is almost faithful so it is the family of graphs of translates of $h$ and then Lemma \ref{existence of infinely many N slopes Lemma} implies that $E_\times (h)$ -- as defined in the statement of Lemma \ref{Lemma infinite N slopes complete} -- is non empty. Moreover, our assumptions imply that $N_\times (h)>1$, so we may apply Lemma \ref{Lemma infinite N slopes complete} and use the fact that $Y_{a,c}$ contains the graph of $h_a\circ h_c^{-1}$ to conclude.
    \end{proof}

We are ready to prove Proposition \ref{exists additive group configuration}. We construct in $\CG$ a group configuration inter-algebraic with a standard group configuration for $(K,+)$. The construction is very similar to the one carried out in Proposition \ref{proposition field given infinitely slopes}. The idea is that if $\CY$ is a ($\CG$-definable) family of plane curves with infinitely many $N$-slopes at $(1,1)$ then, if $N>1$, composition of curves in $\CY$ acts on the $N$-th coefficients of curves in $\CY$ by addition.

\begin{proof}[Proof of Proposition \ref{exists additive group configuration}]
     Let $V$ and $c$ be as provided by Lemma \ref{lemma first order multiplicative} and for $a\in V$ put $Y_a=t_a(C)\circ t_c(C)^{-1}$ and let $\CY=(Y_a)_{a\in C}$.
 
 Let $(f_a)_{a\in V\cap C}$ be a family of functions witnessing that $(Y_a)_{a\in V\cap C}$ has infinitely many $N$ slopes at $(1,1)$ and let $B$ be an open ball contained in the set of $N$-slopes of that family.
 
 Let $t\in B$, then there are generic independent elements $a_1, a_2\in B-t$ such that $a_1+a_2=a_3 \in B-t$.

    We can find $s_{t},s_{a_1}$ and $s_{a_2}\in V\cap C$ such that $f^{(N)}_{s_t}(1)=t$ $f^{(N)}_{s_{a_1}}(1)=a_1+t$ and $f^{(N)}_{s_{a_2}}(1)=a_2+t$. 

    In the same way, we may find $s_{a_3}\in V\cap C$ such that $f^{(N)}_{s_{a_3}}(1)=a_3+t$. By Lemma \ref{compositionMultLemma} it follows that 
   \begin{equation}\label{equation N slopes equals}
   (f_{s_{a_1}}\circ f_{s_{a_2}})^{(N)}(1)=a_1+a_2+t+t=a_3+t+t=(f_{s_{a_3}}\circ f_{s_{t}})^{(N)}(1).\end{equation}

    Moreover, there is $c_1\in B-t$ is generic over all the data such that $a_1+a_2+c_1=c_2$ and $a_2+c_1=c_3$ both belong to $B-t$. So 
    \begin{equation}\tag{**}\label{group configuration additive}
        \{a_1,a_2,a_3,c_1,c_2,c_3\}\sub B-t
    \end{equation}
    
    is a standard group configuration for $\mathbb G_a$. Thus for $i=1,2,3$ we may find $s_{c_i}\in V\cap C$ such that $f^{(N)}_{s_{c_i}}(1)=c_i+t$. 

    Then \[\{s_{a_1},s_{a_2},s_{a_3},s_{c_1},s_{c_2},s_{c_2}\}\] is a group configuration over $t$ equivalent to (\ref{group configuration additive}). It will suffice to show that it is a group configuration over $s_t$ for $\mathcal G$.

     Notice that the family $\CY\circ Y_{s_t}$ has Morley Rank  $1$ and, by the comment at the end of Section \ref{ss plane curves}, we may assume that $\CY\circ \CY$ has Morley Rank greater than $1$. Therefore, as $Y_{s_{a_1}}\circ Y_{s_{a_2}}$ is a generic element of $\CY\circ \CY$ and $Y_{s_{a_3}}\circ Y_{s_t}$ is generic in $\CY\circ Y_{s_t}$, it follows from Remark \ref{R: finite intersection} that $ Y_{s_{a_3}}\circ Y_{s_t} \cap  Y_{s_{a_1}}\circ Y_{s_{a_2}}$ is finite.
     
     In addition, by Equation \ref{equation N slopes equals} \[(f_{s_{a_1}}\circ f_{s_{a_2}})^{(N)}(1)=(f_{s_{a_3}}\circ f_{s_{t}})^{(N)}(1),\] so $Y_{s_{a_1}}\circ Y_{s_{a_2}}$ and $Y_{s_{a_3}}\circ Y_{s_t}$ have the same $N$-order slope at $(1,1)$ and then Corollary \ref{C: definability of tangency} implies that there are infinitely many $s\in V\cap C$ such that $$\mid Y_s\circ Y_{s_t} \cap  Y_{s_{a_1}}\circ Y_{s_{a_2}} \mid > \mid Y_{s_{a_3}}\circ Y_{s_t} \cap  Y_{s_{a_1}}\circ Y_{s_{a_2}} \mid.$$ 
    As $C$ is strongly minimal the set of $s\in S$ such that $$\mid Y_s\circ Y_{s_t} \cap  Y_{s_{a_1}}\circ Y_{s_{a_2}} \mid = \mid Y_{s_{a_3}}\circ Y_{s_t} \cap  Y_{s_{a_1}}\circ Y_{s_{a_2}} \mid$$ is finite so after adding $s_t$ to the parameters $s_{a_3}\in \acl_{\mathcal G}(s_{a_1},s_{a_2})$. The rest of the algebraic relations needed for the group configuration follows by the same argument.
\end{proof}

   Finally, we prooceed to the proof of Proposition \ref{proposition multiplicative case}.

    \begin{proof}[Proof of Proposition \ref{proposition multiplicative case}]

   By the discussion following \ref{D: strongly t-an}, after removing finitely many points of $C$ we may assume that $C$ and $C^{-1}$ are both strongly $t$-analytic. We consider the almost faithful family $\{t_a(C):a\in C\}$.  If the family of translates $\{t_a(C):a\in C\}$ has infinitely many slopes at $(1,1)$ then Proposition \ref{proposition field given infinitely slopes} implies the existence of an interpretable field. 
    
    Otherwise, Proposition \ref{exists additive group configuration} provides a group configuration in $\mathcal G$ inter-algebraic with a standard group configuration for the additive group. By Hrushovski's group configuration there is a  $\mathcal G$-interpretable strongly minimal group $(H,\otimes)$ with a group configuration  inter-algebraic with a group configuration of the additive group of $K$. By Theorem \ref{thmFiniteIndex} there is a  $\mathcal G$-interpretable strongly minimal group $(G,\oplus)$  containsing a subgroup $G_0$ of finite index and a finite subgroup $F$ such that  $G_0/F$ is isomorphic to a subgroup of the additive group. As $G$ is commutative the group $G/F$ is $\mathcal G$ definable and the $\CG$-induced structure  on $(G/F,\oplus)$ is strongly minimal and non locally modular. Thus, replacing $G$ with $G/F$  we may assume that $F$ is trivial.  Notice that then $G$ has a finite index subgroup isomorphic to a subgroup of the additive group and the $\CG$-induced structure on $(G,\oplus)$ is strongly minimal and non locally modular, so by Proposition \ref{proposition additive case} there is a $\mathcal G$-interpretable field which is $\mathcal K$-definably isomorphic to $\mathcal K$.
    \end{proof}

\section{Conclusion}\label{POTM}
	We sum up everything we did to provide a complete proof of the main result of the paper. For the reader's convinence, we restate the theorem: 
 \begin{thm*}
     Let $\CK$ be an algebraically closed valued field, and  $\mathcal G:=(G,\oplus, \dots)$ a non-locally modular strongly minimal $\CK$-group-relic. Assume further that $G$ is locally equivalent to either $(K,+)$ or to $(K^*,\cdot)$.  Then $\mathcal G$ interprets a field, $F$, $\CK$-definably isomorphic to $(K,+,\cdot)$. Furthermore, the $\mathcal G$-induced structure on $F$ is that of a pure algebraically closed field.
\end{thm*}

	\begin{proof}
				
    Since $G$ is strongly minimal $G^0$, its stability theoretic connected component, exists and is a definable subgroup. So replacing $G$ with $G^0$ we  may assume that $G$ has no $\CG$-definable subgroups of finite index.  By the reduction of Section \ref{ss: definable}, we may assume that $G$ is definable. 
				
    Since $G$ is strongly minimal, it is abelian. Replacing $G$ with a quotient by a finite subgroup, if needed, we conclude  by Theorem \ref{thmFiniteIndex}, that there exists a $\CK$-definable subgroup $G_0\le G$ of finite index that is  $\CK$-definably isomorphic to either a subgroup of the additive group, or to a subgroup of the multiplicative group or to $\mathbb G_{m,x}$. Consequently, we  may apply the construction of Section \ref{ss: the setting}. In particular $G$ is locally isomorphic to either $(K,+)$ or to $(K^*, \cdot)$. 
				
	In the former case, Proposition \ref{proposition additive case} completes the proof of the theorem.
				
    In the other cases, if $n=|G/G_0|$  consider the map $\oplus_n$ sending $g$ to $$\underbrace{g\oplus g \dots \oplus g}_{n\text{ times}}.$$ Since $G_0$ is isomorphic to a subgroup of the multiplicative group or to $\mathbb G_{m,x}$ it follows that $\oplus_n$ has finite kernel. Since $G$ has no $\CG$-definable subgroups of finite index the image of $\oplus_n$ is the whole of $G$. Since the image of $\oplus_n$ is contained in $G_0$, it follows that $G_0=G$. implying that $G$ is either a subgroup of the multiplicative group or $\mathbb G_{m,x}$. Proposition \ref{proposition multiplicative case} completes the construction of the field in all cases. 
				
	To complete the proof of the theorem, it remains to show that the field $F$ constructed up until now is $\CK$-definably isomorphic to $K$. Since $\CG$ is strongly minimal and $F$ is interpretable in $\CG$ there is a finite-to-finite $\CG$-definable correspondence between $K$ and $G$. As we have shown in Section \ref{ss: definable} since $G$ may be assumed to be definable, there is a finite-to-finite correspondence between $F$ and a definable subset of $K^n$ (some $n$). In particular $\dim(F)=1$ (where $\dim$ is Gagelman's extension of dimension to $\CK^{eq}$, as discussed in Section \ref{ss: definable}). By \cite[Theorem 6.23]{HruRid} $F$ is $\CK$-definably isomorphic to either $K$ or the residue field $\mathbf{k}$. Since $\dim(\mathbf{k})=0$, and dimension is preserved under definable bijections, it must follow that $F$ is definably isomorphic to $K$. 
				
	To show that $F$ is a pure field, we use an argument suggested to us by B. Castle. As already noted, $\CK$ is algebraically bounded (i.e, the model theoretic $\acl$ is the field theoretic algebraic closure). So the model theoretic algebraic closure in $F$ is the field theoretic algebraic closure. By \cite[Theorem 1]{HR2} any strongly minimal $\acl$-preserving expansion of an algebraically closed field, is an expansion by algebraic constants, with the desired conclusion.	\end{proof}
			
    We conclude by recalling that the full trichotomy for strongly minimal sets definable in ACVF is now known (by reducing to the results of the present paper).

			\bibliographystyle{alpha}
			\bibliography{biblio}
		\end{document}